\documentclass[11pt]{amsart}
\usepackage{amsfonts,a4wide}
\usepackage{mathrsfs}
\usepackage{amssymb}
\usepackage{graphicx}
\usepackage{mathtools}
\usepackage{enumerate}
\usepackage{color}
\newtheorem{theorem}{Theorem}
\newtheorem{corollary}[theorem]{Corollary}
\newtheorem{lemma}[theorem]{Lemma}
\newtheorem{proposition}[theorem]{Proposition}
\theoremstyle{definition}

\newtheorem{remark}[theorem]{Remark}
\newtheorem{example}[theorem]{Example}

\newcommand{\IM}{\mbox{\rm Im\,}}
\newcommand{\RE}{\mbox{\rm Re\,}}

\DeclareMathOperator{\HHH}{H}
\newcommand{\hau}{d_{\HHH}}

  \DeclareMathOperator{\dom}{dom}
\DeclareMathOperator{\ii}{i}

\DeclareMathOperator{\p}{p}
\DeclareMathOperator{\Int}{Int}
\DeclareMathOperator{\conv}{conv}
\DeclareMathOperator{\dist}{dist}
\DeclareMathOperator{\ap}{ap}

\newcommand{\Real}{\mathbb{R}}
\newcommand{\Comp}{\mathbb{C}}
\newcommand{\Nat}{\mathbb{N}}
\newcommand{\Disc}{\mathbb{D}}

\newcommand{\eps}{\varepsilon}

\newcommand{\h}{\mathcal{H}}
\newcommand{\Kk}{\mathcal{K}}

\newcommand{\cB}{\mathcal{B}}

\newcommand{\set}[1]{\left\{#1\right\}}
\newcommand{\seq}[1]{\left<#1\right>}
\newcommand{\norm}[1]{\left\Vert#1\right\Vert}

\newcommand{\rest}[1]{\!\!\mid_{#1}}

 \providecommand{\seq}[1]{\left<#1\right>}
\providecommand{\norm}[1]{\left\Vert#1\right\Vert}
\newcommand {\mycomment}[1]{} % has the effect of %%%-ing text out
\newcommand {\mat}  [1] {\left[\begin{array}{#1}}
\newcommand {\rix}      {\end{array}\right]}

\catcode`@=11     % Make @ act like a letter
\font\tenex=cmex10 % math extension
\newdimen\p@renwd
\setbox0=\hbox{\tenex B} \p@renwd=\wd0 % width of the big left [
\def\bmat#1{\begingroup \m@th
  \setbox\z@\vbox{\def\cr{\crcr\noalign{\kern2\p@\global\let\cr\endline}}%
    \ialign{$##$\hfil\kern2\p@\kern\p@renwd&\thinspace\hfil$##$\hfil
      &&\quad\hfil$##$\hfil\crcr
      \omit\strut\hfil\crcr\noalign{\kern-\baselineskip}%
      #1\crcr\omit\strut\cr}}%
  \setbox\tw@\vbox{\unvcopy\z@\global\setbox\@ne\lastbox}%
  \setbox\tw@\hbox{\unhbox\@ne\unskip\global\setbox\@ne\lastbox}%
  \setbox\tw@\hbox{$\kern\wd\@ne\kern-\p@renwd\left[\kern-\wd\@ne
    \global\setbox\@ne\vbox{\box\@ne\kern2\p@}%
    \vcenter{\kern-\ht\@ne\unvbox\z@\kern-\baselineskip}\,\right]$}%
  \null\;\vbox{\kern\ht\@ne\box\tw@}\endgroup}
\catcode`@=12    % @ is no longer a letter
\newif\ifMDlatex
\catcode`@=11     % Make @ act like a letter
\def\MD@us#1{\csname#1style\endcsname}
\def\MD@uf#1{\csname#1font\endcsname}
\def\MD@t{text}
\def\MD@s{script}
\def\MD@ss{scriptscript}
\newdimen\MD@unit
\MD@unit\p@
\def\MD@changestyle#1{
  \relax\MD@unit0.1\fontdimen6\MD@uf{#1}0
  \everymath\expandafter{\the\everymath\MD@us{#1}}
}
\def\MD@dot{$\m@th\ldotp$}
\def\MD@palette#1{\mathchoice{#1\MD@t}{#1\MD@t}{#1\MD@s}{#1\MD@ss}}
\def\MD@ddots#1{{\MD@changestyle{#1}%
  \mkern1mu\raise7\MD@unit\vbox{\kern7\MD@unit\hbox{\MD@dot}}%
  \mkern2mu\raise4\MD@unit\hbox{\MD@dot}%
  \mkern2mu\raise \MD@unit\hbox{\MD@dot}\mkern1mu}}%
\def\MD@iddots#1{{\MD@changestyle{#1}%
  \mkern1mu\raise \MD@unit\hbox{\MD@dot}%
  \mkern2mu\raise4\MD@unit\hbox{\MD@dot}%
  \mkern2mu\raise7\MD@unit\vbox{\kern7\MD@unit\hbox{\MD@dot}}}}%
\def\MD@vdots#1{\vbox{\MD@changestyle{#1}%
    \baselineskip4\MD@unit\lineskiplimit\z@
    \kern6\MD@unit\hbox{\MD@dot}\hbox{\MD@dot}\hbox{\MD@dot}}}%
\ifMDlatex
  \DeclareRobustCommand\ddots{\mathinner{\MD@palette\MD@ddots}}%
  \DeclareRobustCommand\iddots{\mathinner{\MD@palette\MD@iddots}}%
  \DeclareRobustCommand\vdots{\mathinner{\MD@palette\MD@vdots}}%
\else
  \def\ddots{\mathinner{\MD@palette\MD@ddots}}%
  \def\iddots{\mathinner{\MD@palette\MD@iddots}}%
  \def\vdots{\mathinner{\MD@palette\MD@vdots}}%
\fi
\catcode`@=12    % @ is no longer a letter

\newcommand{\matp}[1]{\begin{bmatrix} #1 \end{bmatrix}}

\usepackage{color}

\begin{document}
\title
{Between the von Neumann inequality and the Crouzeix conjecture. }
\author{Patryk Pagacz, Pawe{\l} Pietrzycki, and Micha{\l} Wojtylak
}

  \address{Faculty of Mathematics and Computer Science, Jagiellonian University, {\L}ojasiewicza
6, 30-348 Krakow, Poland}
   \email{\{patryk.pagacz, pawel.pietrzycki, michal.wojtylak\}@im.uj.edu.pl}
\thanks{MW gratefully acknowledges the support of the Humboldt Foundation with a short research stay grant}

\begin{abstract}
    A new  concept of a deformed numerical range $W^\rho(T)$ is introduced. Here $T$ is a bounded linear operator or a matrix and $ \rho \in[1,+\infty)$ is a parameter. Each $W^\rho(T)$ is a closed convex set that contains the spectrum of $T$.
    Furthermore,  $W^\rho(T)$ is decreasing with respect to $ \rho $ and $W^2(T)$ coincides with the numerical range.
    It is also shown that $W^\rho(T)$ is contained in the closed unit disc if and only if $T$ has a $\rho$ unitary dilation in the sense of N\'agy-Foia\c s.
    The spectral constants of $W^\rho(T)$ is investigated, it is shown that it is monotone and continuous with respect to the parameter $ \rho $.
\end{abstract}
\maketitle

\section{Introduction}

%\cite{BCK,ransford2018remarks,crouzeix2017numerical,caldwell2017some,delyondelyon1999,faber2010chebyshev,crouzeix2004bounds,okubo1975constants,nagyfoias1966,GreenChoiLAA,ChoiLAA,GreenCaldwellLi,GreenLewisOverton,
%BergerStampfli,DavidsonPaulsen,KlajaJavad}

The celebrated von Neumann inequality states  that if $T$ is a bounded operator on a complex Hilbert space $\h$ and $p(z)$ is a polynomial then
\begin{equation}\label{vonN}
\norm{p(T)}\leq \sup_{z\in\norm T\cdot\overline \Disc} |p(z)|,
\end{equation}
where  $\Disc$ stands for the open unit disc.
The second seminal result of interest  is the following one
\begin{equation}\label{Cro2}
\norm{p(T)}\leq \Psi_1(T)\sup_{z\in W(T)} |p(z)|,
\end{equation}
 where by  $W(T)$ we denote the numerical range of $T$, i.e.,
$$
W(T)=\{\langle Th,h\rangle: h\in \h\}.
$$
The constant $\Psi_1(T)$ on the right hand side was initially prove to exist in \cite{delyondelyon1999}.  Crouzeix in \cite{Crouzeix2007} established that $\Psi_1(T)\leq 11.08$ and conjectured that $\Psi_1(T)\leq 2$ for any bounded operator $T$.  The conjecture is true for $2\times 2$ matrices (see \cite{crouzeix2004bounds}) and a simple $2\times 2$ matrix  example show that the constant $2$ is the best possible.
Up to now the proof of Crouzeix conjecture is know only for some special cases (see \cite{ChoiLAA,GreenChoiLAA}). The current best estimate $\Psi_1(T)\leq 1+\sqrt 2$   was obtained   by  Crouzeix and  Palencia
 in \cite{crouzeix2017numerical}, see also Ransford and Schwenninger \cite{ransford2018remarks}.

Usually one expresses the  inequality \eqref{vonN} by saying that  the disc of radius $\norm T$ is a $1$-spectral set, analogously   the numerical range is a $(1+\sqrt 2)$-spectral set (cf. \eqref{Cro2}). The goal of the present paper is to construct intermediate spectral sets. For this aim we define the deformed numerical range $W^\rho(T)$ as the closed convex hull of
$$
\set{ \xi_{\rho}(h) \seq{Th,h}:\norm h=1},\quad  \rho \in[1,2],
$$
where
%\begin{equation}\label{xiintro}
%\xi_{\rho}(h)=\begin{cases}
%                     \frac12\left(   r  +    \frac{\sqrt{ r ^2|\seq{Th,h}|^2-4( r -1)\norm{Th}^2} }{|\seq{Th,h}|}\right) & \mbox{: } \seq{Th,h}\not=0 \\
%                     0 & \mbox{: }  \seq{Th,h}=0.
%                   \end{cases}
%                   \end{equation}

\begin{equation}\label{xiintro}
\xi_{\rho}(h)=\begin{cases}
                         1-\frac1\rho  +    \frac{\sqrt{ (1-\frac1\rho)^2|\seq{Th,h}|^2-( 1-\frac2\rho)\norm{Th}^2} }{|\seq{Th,h}|} & \mbox{: } \seq{Th,h}\not=0 \\
                    {\color{black} 1 }& \mbox{: }  \seq{Th,h}=0.
                   \end{cases}
                   \end{equation}

The definition for $ \rho \in[2,+\infty)$ is more technical, see Section \ref{def} for details.  Note that $W^2(T)$ equals (the closure of) the numerical range. Our main results concerning these sets are the following. In Theorem \ref{basic}
we show that the spectrum of $T$ is contained in $W^\rho(T)$, in Theorem \ref{monocont} we show that $W^\rho(T)$ is monotone and continuous in the Hausdorff metric with respect to $ \rho \in[1,+\infty)$. Furthermore, we show that there exists a constant $\Psi_\rho(T)$ such that
$$
\norm{p(T)}\leq \Psi_\rho(T)\sup_{z\in W^\rho(T)} |p(z)|,
$$
and that $\rho\mapsto \Psi_\rho(T)$ is continuous with respect to $\rho\in[1,+\infty)$, see Theorem \ref{qcont}.

Another direction of research is to analyse spectral constants connected with simple
regions containing the numerical range.
E.g,  it is known since \cite{okubo1975constants} that for any polynomial $p$
\begin{equation}\label{Cro1}
\norm{p(T)}\leq 2\sup_{\nu(T)\cdot \overline\Disc} |p|,
\end{equation}
where $\nu(T)$ stands for the numerical radius of $T$, i.e.
$ \nu(T):=\sup_{z\in W(T)} |z|.
$
In fact, the above inequality  holds with any disc containing $W(T)$ in place of $\nu(T)\overline\Disc$. % see also  \cite{badea2006convex} for partial results leading to \eqref{Cro1}.
 In our paper we show that
\begin{equation}\label{Croq}
\norm{p(T)}\leq\rho \sup_{\nu_\rho(T)\cdot \overline\Disc} |p|,
\end{equation}
where $\nu_\rho(T)$ stands for the deformed spectral radius of $T$, i.e.
$ \nu_\rho(T):=\sup_{z\in W^\rho(T)} |z|$ (cf. \cite{badea2018spectral}). This constitutes  a continuous passage between \eqref{vonN} and \eqref{Cro1}.
Our work in this direction was motivated other recent methods for obtaining spectral constants for discs by  Caldwell, Greenbaum and Lie \cite{caldwell2018some} and for annuli by Crouzeix and Greenbaum \cite{crouzeix2019spectral}.

Most of the results we obtained are true both for complex matrices and for bounded operators on Hilbert spaces, however, some of them concern only one of these classes. This influences the organisation of the paper.  And so, in Section \ref{def} we define the deformed numerical range and show its basic properties for matrices and operators. The main outcome,  the inclusion $\sigma(T)\subseteq W^\rho(T)$, is showed in the matrix case only.
Then, in Section \ref{simcase} we give  examples of $W^\rho(T)$ for $2\times 2$ matrices. The next part, Section \ref{Operators}, is devoted to quasinilpotent operators in a Hilbert space. Analysis of this class is necessary to complete the proof  of $\sigma(T)\subseteq W^\rho(T)$ in the operator case. Subsequently, in Section \ref{NF} we return to the general setting and show the connection of $W^\rho(T)$ with  the N\'agy-Foia\c s dilation theory.   In Section \ref{smono} we show the announced monotonicity and continuity of $W^\rho(T)$ with respect to $\rho$. This is applied in  Section \ref{s:specconst} to analyse the spectral constants $\Psi_\rho(T)$. In the last Section we present the relation of $W^\rho(T)$ to the known concepts from dilation theory: Mathias-Okubo formula for $\rho$-numerical range, the $q$-numerical range, normalised numerical range and Davies-Wielandt shell.

The following notation will be used. The
fields of real and complex
numbers are denoted by
$\mathbb{R}$ and $\mathbb{C}$, respectively. % The symbols $\mathbb{Z}$, $\mathbb{Z}_{+}$, $\mathbb{N}$
%and $\mathbb{R}_+$ stand for the sets of integers,
%positive integers, nonnegative integers,  and
%nonnegative real numbers, respectively.
All Hilbert spaces considered in this paper are assumed to be complex, $\seq{f,g}$ and $\norm f$ stands for the inner product and the corresponding norm, respectively.
In the finite dimensional case we assume that $\mathcal H=\Comp^n$ and $\seq{x,y}$ stands for the usual inner product.
We denote by $\sigma(T)$ the spectrum of $T$ and by $\sigma_{\p}(T)$ and $\sigma_{\ap}(T)$ we mean the point spectrum (eigenvalues) and approximative spectrum of $T$, respectively.  Furthermore,  as the letters $\rho$ and $r$ are already reserved, we denote the spectral radius of $T$ by $\nu_\infty(T)$. By $\overline{V}$, $\Int V$ and $\partial V$ we mean the closure, the interior and the boundary of $V\subseteq\Comp$.
We will use without further notice the fact that one may interchange the order of taking the convex hull and the closure, i.e. $\conv \overline V=\overline{\conv V}$, for any bounded subset $V$ of the complex plain.
As it was already used, $r\Disc$ is an open disc centred at origin with radius $r$.

\section{Deformed numerical range: definition and basic properties}\label{def}
Let us begin with defining  the main objects. For $ \rho\in[1,+\infty)$ and $h\in\h$ with $\norm h=1$  we set
\begin{equation}\label{Deltadef}
\Delta_\rho(h):= r ^2|\seq{Th,h}|^2-4( r -1)\norm{Th}^2,\quad r=2\left(1-\frac1\rho\right).
\end{equation}
Here and in the sequel the auxiliary variable $r$ was introduced for better clarity of the formulas. Note that the mapping $[1,+\infty)\ni\rho\mapsto r\in[0,2)$ is a increasing bijection and  $\rho=2$ is mapped to $r=1$. Later on we see that $\rho=2$ corresponds to the numerical range. 
We denote the domain of the function $\xi_\rho$ as
\color{black}
\begin{equation}\label{dom_xi}
  \dom(\xi_\rho):=\{h\in\h :   \|h\|=1,\  ( \Delta_\rho(h) > 0 \text{ or } Th=0) \}.
\end{equation}
\color{black}
and we set
\color{black}
\begin{equation}\label{xidef}
\xi_\rho(h):=\begin{cases} \frac12\left( r +\frac{\sqrt{\Delta_ r (h)}}{|\seq{Th,h}|}\right) &: \seq{Th,h}\neq0\\
1 &: \seq{Th,h}=0
\end{cases},\quad h\in\dom(\xi_\rho),\ \rho\in[1,+\infty).
\end{equation}\color{black}

Note that   if $\seq{Th,h}=0$ and $Th\neq 0$ then $h\in\dom(\xi_\rho)$ precisely for $ \rho \in[1,2]$. Such vectors $h$ will require separate treatment in the course of the paper, especially in  Theorems \ref{basic} and \ref{Wdysk}.
Observe also that the definition of $\xi_\rho(h)$  in \eqref{xiintro} the Introduction  coincides with  \eqref{xidef} above, we will use  the  latter one for further reasoning. We now list the basic properties of the  functions $\xi_\rho$ and $\Delta_\rho$.

\begin{proposition}\label{basicxi} For any bounded operator $T$ on a Hilbert space the following holds.
\begin{enumerate}[\rm (i)]
\item\label{(i)} If $1\leq \rho_1\leq \rho_2 < +\infty$ then $\dom(\xi_{\rho_2})\subseteq \dom(\xi_{\rho_1})$ and $\xi_{\rho_2}(h)\leq \xi_{\rho_1}(h)$, for $h\in\dom(\xi_{\rho_2})$.
\item\label{(i')} If $0\leq \rho_1 < \rho_2 < 2$ and $h\in\dom(\xi_{\rho_2})$ with $\seq{Th,h}\neq 0$, then  $\xi_{\rho_2}(h)= \xi_{\rho_1}(h)$ if and only if $h$ is an eigenvector of $T$.
\item\label{(ii)} $\dom (\xi_\rho)=\set{h\in\h:\norm h=1}$ for $\rho\in[1,2]$. %  $\Delta_ r (h)> 0$ for $ r \in(0,1)$, $Th\neq0$, $h\in\h$.
\item\label{(iii)} $\xi_1(h)=\frac{\norm{Th}}{|\seq{Th,h}|}$, $\xi_2(h)=1$ for $h\in\h$ with $\norm h=1$, $\seq{Th,h}\neq 0$.
\color{black}
\item\label{(iv)} If $Th=\lambda h$ for a unit $h$ and  some $\lambda\in\Comp $  then $h\in\dom(\xi_\rho)$ and $\xi_\rho(h)=1$ for $\rho\in[1,+\infty)$. 
\item\label{(v)} \color{black} If $Th\neq\lambda h$ for all $\lambda\in\Comp$ and $\norm h=1$ then $h\notin\dom(\xi_\rho)$ for $\rho\in[\rho_0,\infty)$ for some $\rho_0\in[2,+\infty)$.
%\item[(iv)] $1\leq \xi_\rho(h)\leq \frac{\norm T}{|\seq{Th,h}|}$ for $h\in\dom(\xi_\rho)$, $ r \in[0,1]$.
%\item[(v)]  $\xi_\rho(h)\leq1$ for $ r \in(1,2)$ and for $h\in\dom(\xi_\rho)$.
\end{enumerate}
\end{proposition}

\begin{proof}
For a fixed unit $h$ we treat $\Delta_\rho(h)$ and $\xi_\rho(h)$ as functions of the parameter $r=2(1-\rho^{-1})$, see \eqref{Deltadef} above.  The first part of statement \eqref{(i)} follows from  the fact that  for a fixed $h$ with $\norm h=1$ we have
$$
\frac{d \Delta_\rho (h)}{d  r } = 2 ( r \ |\seq{Th,h}|^2-2\norm{Th}^2) \leq 0,\quad  r \in[0,2).
$$
Now let $h\in\dom(\xi_{\rho_2})$ and let $\seq{Th,h}\neq 0$.
An elementary calculation shows that
$$
\frac{d \xi_\rho(h)}{d  r }  \leq 0,\quad  r \in[0, r _2),
$$
and the equality holds if and only if $|\seq{Th,h}|=\norm{Th}$. This shows the second part of \eqref{(i)} and \eqref{(i')}.

Statements \eqref{(ii)} and \eqref{(iii)} are obvious.
\color{black}
To see \eqref{(iv)} note that $\Delta_\rho(h)=|\lambda|^2(2- r )^2 > 0$, for $\lambda\not=0$. As $ r \in[0,2)$, we have clearly $\xi_\rho(h)=1$ for $\rho\in[1,+\infty)$. For $Th=0$ \eqref{(iv)} follows from directly \eqref{dom_xi} and \eqref{xidef}.
\color{black}
Let us now show \eqref{(v)}. Take $h$ as in the statement, then $|\seq{Th,h} | < \norm{Th}$ and by elementary expression
\begin{equation}\label{DeltaRozpisana}
\Delta_ \rho (h)=(2- r )^2|\seq{Th,h}|^2-4( r -1)(\|Th\|^2-|\seq{Th,h}|^2)
\end{equation}
 we have $\Delta_ \rho (h)\leq 0$ for $ r \in[ r _0,2)$ for some $ r _0\in[1,2)$, which is exactly the claim.
\end{proof}

Further for $\rho\in[1,+\infty)$ we define the {\em deformed numerical range} of a nonzero operator $T$ as
\begin{equation}\label{dns}
W^\rho(T)=\conv \overline{\set{ \xi_\rho(h) \seq{Th,h}: h\in\dom(\xi_\rho) }}
\end{equation}
and \emph{the deformed numerical radius} as
\begin{equation}\label{dnr}
\nu_\rho(T)=\sup_{z\in W^\rho(T)}|z|.
\end{equation}

Observe that, almost trivially,
\begin{equation}\label{W2W}
W^2(T)=\overline{W(T)},
\end{equation}
where $W(T)$ stands for the numerical range of $T$.
Theorem \ref{basic} shows the main properties of the deformed numerical range.

\begin{theorem}\label{basic} For a bounded operator $T$  on a Hilbert space $\h$ and $ \rho\in[1,+\infty) $ the following holds.
\begin{enumerate}[\rm (i)]
\item\label{0} {\color{black} If $T=T^*$ then $W^\rho(U^*TU)\subseteq\Real$;}
\item\label{U} $W^\rho(U^*TU)=W^\rho(T)$  for any unitary operator $U$ on $\h$;
\item\label{alpha} $W^\rho(\alpha T)=\alpha W^\rho(T)$ for any $\alpha\in\Comp$;
\item\label{K} if $\Kk$ is a subspace of $\h$ invariant for $T$, then $W^\rho(T\rest\Kk)\subseteq W^\rho(T)$;
%\item\label{clo} if $\dim\h<\infty$ then
%\begin{equation*}
%    W^\rho(T)=\conv \set{ \xi_\rho(h) \seq{Th,h}: h\in\dom(\xi_\rho) }
%\end{equation*}
\item\label{sigmap} $ W^\rho(T)\subseteq \norm T \overline\Disc$;
\item\label{sigma} $\sigma(T)\subseteq W^\rho(T)$.

\end{enumerate}
\end{theorem}

Statements \eqref{0}, \eqref{U}, \eqref{alpha}, and \eqref{K} are elementary, \eqref{sigmap}  follows directly from statements \eqref{(i)} and \eqref{(iii)} of Proposition \ref{basicxi}. Also it follows directly from Proposition \ref{basicxi}\eqref{(iv)} that
the eigenvalues  are contained in $W^\rho(T)$ for any $ \rho\in[1,+\infty) $, hence \eqref{sigma} is showed if $T$ is a matrix. The proof in the operator case will be completed in Section \ref{Operators} and requires some additional preparation concerning quasinilpotent operators. Now let us study the simplest examples and instances.

%\eqref{clo}

\begin{remark}\label{trans}
  In addition to Theorem  \ref{basic}\eqref{alpha} note that, except the case $ \rho=2$,  $W^\rho(T+\alpha I)$  is in general not equal to $W^\rho(T)+\alpha$.
   This can be seen in various ways, we present here a general reason in case when $T$ is a matrix and $ \rho\in[1,2] $. Note that for a matrix $T$ the deformed numerical range $W^\rho(T)$  is contained in the closed right half-plane  if and only if $W^2(T)=W(T)$ has this property, hence
for a fixed $  \rho\in[1,2] $ the set $ W^\rho(T)$ is contained in the closed right half-plane if and only if $T+T^*\geq 0$.
    Furthermore,  $W^\rho(T)$ is compact and convex. Hence, if
 $W^\rho(T+ \alpha I)=W^\rho(T)+ \alpha$ for any matrix $T$ and any $\alpha\in\Comp$, then, by \cite[Theorem 1.4.2]{HorJ85}, $W^\rho(T)=W(T)$, which clearly is a contradiction with the definition of $W^\rho(T)$, see e.g.  Theorem \ref{monocont} below.
\end{remark}

\begin{remark}\label{closures}
\color{black}
Let $T$ be a nonzero complex square matrix. Then
\begin{equation}\label{W2W}
W^\rho(T)=\conv\left(\set{\xi_\rho(h)\seq{Th,h}:\norm h=1,\ \Delta_\rho(h)\geq0}\right)
\end{equation}
for all $\rho$  from $[1,2]$ and for all $\rho$ from $(2,+\infty)$ except a finite set.

First note that the set on the right hand side of \eqref{W2W} is closed, by compactness of the unit sphere. The inclusion '$\subseteq$' in \eqref{W2W} is in this light obvious. The converse inclusion for $\rho\in[1,2]$ follows  from Proposition \ref{basicxi}\eqref{(ii)}. The hard work  is  to show the inclusion `$\supseteq$' for $
\rho\in(2,+\infty)$.  For this aim consider the function
$$
G(f):=\frac{|\seq{Tf,f}|^2}{\norm{Tf}^2{\norm f}^2}, \quad \norm f\leq 1,\ Tf\neq 0.
$$
Assume that $\frac{4(r-1)}{r^2}$ is not a weak local maximal value of $G$ and let $r$ and $\rho$ be related as usual, $\rho=2/(2-r)$.
Take an arbitrary unit vector $h$ for which $\Delta_\rho(h)=0$. If $\seq{Th,h}=0$ then $Th=0$ and, by definition, $\xi_\rho(h)\seq{Th,h}\in W^\rho(T)$. Assume now that $\seq{Th,h}\neq 0$, note that one has $G(h)=\frac{4(r-1)}{r^2}$.
Due to our assumption, for every $\eps>0$ there exists a vector $g_\eps$ with $\norm{g_\eps-h}<\eps$ and $G(g_\eps)> \frac{4(r-1)}{r^2}$. Setting $\tilde g_\eps=g_\eps/\norm{g_\eps}$ we have $G(g_\eps)=G(\tilde g_\eps)$, i.e., $\Delta_\rho(\tilde g_\eps)>0$. Due to 
$\seq{Th,h}\neq 0$ we get that $\xi_\rho(\tilde g_\eps)\seq{T\tilde g_\eps,\tilde g_\eps}\to \xi_\rho(h)\seq{Th,h}$ with $\eps\to 0$.  
Hence, $\xi_\rho(h)\seq{Th,h}\in W^\rho(T)$ and the inclusion `$\supseteq$' in  \eqref{W2W} is shown for this $\rho$.

Seeing $G$ is a real rational function in $2n$ real parameters (real and imaginary coordinates of $h$) we observe that $G$ has only a finite number of weak local maximal values, which finishes the proof of the first statement. 

%Proving or disproving \eqref{W2W} for all $\rho\in[1,+\infty)$ in the operator setting seems to be a challenging task.

\end{remark}

\section{Examples}\label{simcase}

In this section we will deal with $2\times 2$ matrices.

\begin{example} \label{pictures} Figure \ref{pictures1}  shows the  set $\set{ \xi_\rho(h) \seq{Th,h}: h\in\dom(\xi_\rho)   }$ for $\rho\in[1,2]$, recall that the closure of the convex hull of this set is, by definition, the deformed numerical range.  We will discuss the connectivity of $\set{ \xi_\rho(h) \seq{Th,h}: h\in\dom(\xi_\rho)   }$ for $\rho\in[1,2]$ in Subsection \ref{ss:shell}.

Figure \ref{pictures2}  shows the  set $\set{ \xi_\rho(h) \seq{Th,h}: h\in\dom(\xi_\rho)   }$ for $\rho\in[2,+\infty)$.
Note that in many instances the plotted set itself is not convex and for $\rho\in[2,+\infty)$ can happen to be not connected.

Later on, on Section \ref{smono} it will become clear that  the location with respect to the origin plays here the essential role. In particular, note that for a normal matrix $T=\matp{\ii & 0\\ 0 & 1}$ the deformed numerical range $W^\rho(T)$ is not the convex hull of its eigenvalues.
\end{example}

\begin{figure}

\caption{The numerical plot of $\set{ \xi_\rho(h) \seq{Th,h}: h\in\dom(\xi_\rho)   }$ for $ \rho=1$ (blue circles)  $\rho=\frac43$ (red crosses) $ \rho=2$ (numerical range, black dots).  }\label{pictures1}
%$T=\begin{array}{cc}   \ii & 0 \\ 0 1 \end{array} $ (first row),}% $T=\matp{\ii & 0 \\ 1 & 1 }$ (second row), $T=\matp{-1 & 1\\ 0 & %1}$ (third row).%\label{pictures}}
\begin{center}
 \includegraphics[width=200pt]{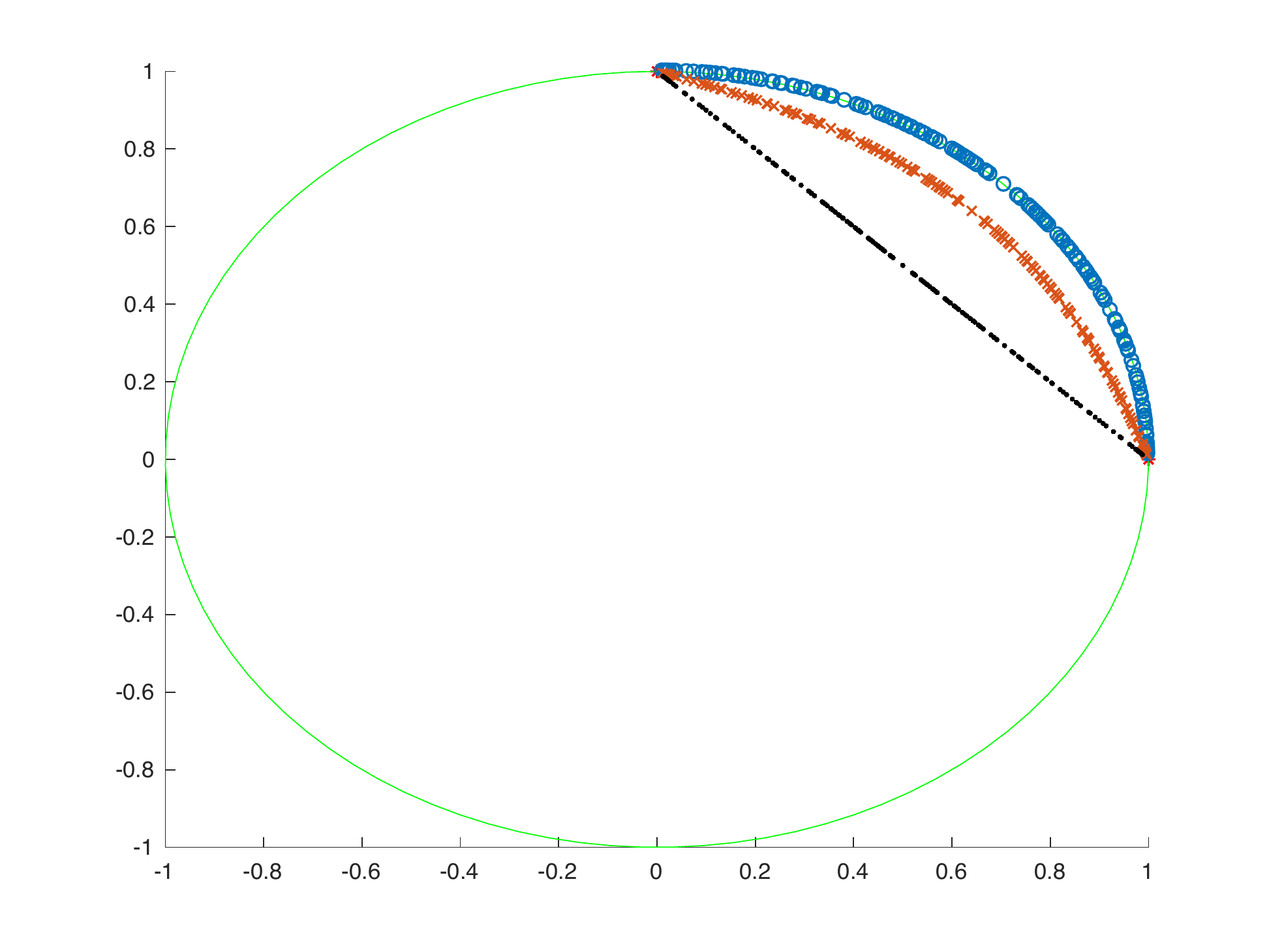}   \\ $T=\matp{   \ii & 0 \\ 0 & 1 }$
\end{center}

\begin{center}
 \includegraphics[width=200pt]{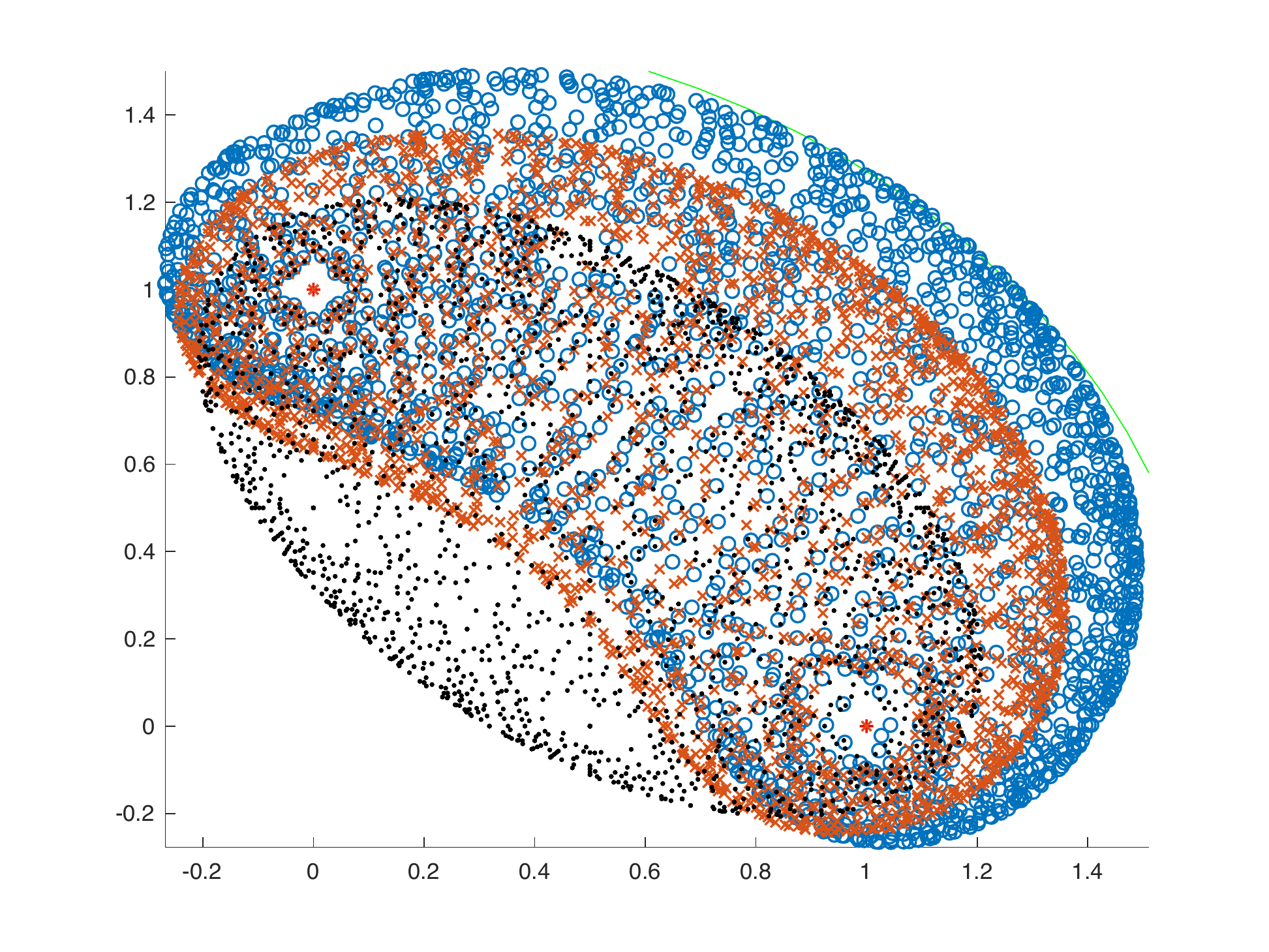} \\ $T=\matp{   \ii & 0 \\ 1 & 1 }$
\end{center}

\begin{center}
 \includegraphics[width=200pt]{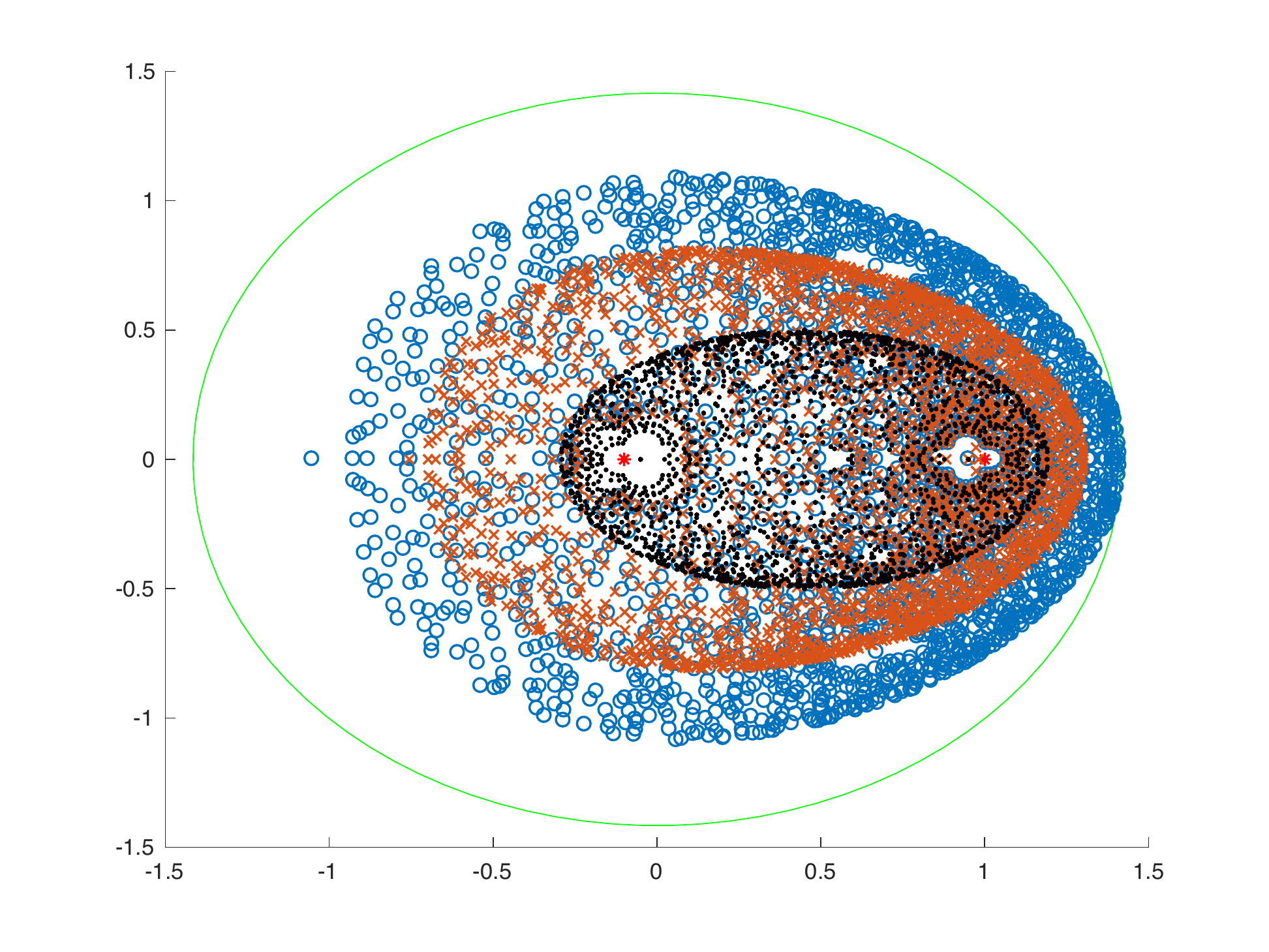} \\ $T=\matp{   -0.1 & 1 \\ 0 & 1 }$
\end{center}

\end{figure}

%$T=\begin{array}{cc}   \ii & 0 \\ 0 1 \end{array} $ (first row),}% $T=\matp{\ii & 0 \\ 1 & 1 }$ (second row), $T=\matp{-1 & 1\\ 0 & %1}$ (third row).%\label{pictures}}

\begin{figure}
\caption{The numerical plot of $\set{ \xi_\rho(h) \seq{Th,h}: h\in\dom(\xi_\rho)   }$ for $ \rho=2$ (numerical range, blue)  $\rho=4$ (orange) $\rho=20$ (black).  }\label{pictures2}

\begin{center}
 \includegraphics[width=200pt]{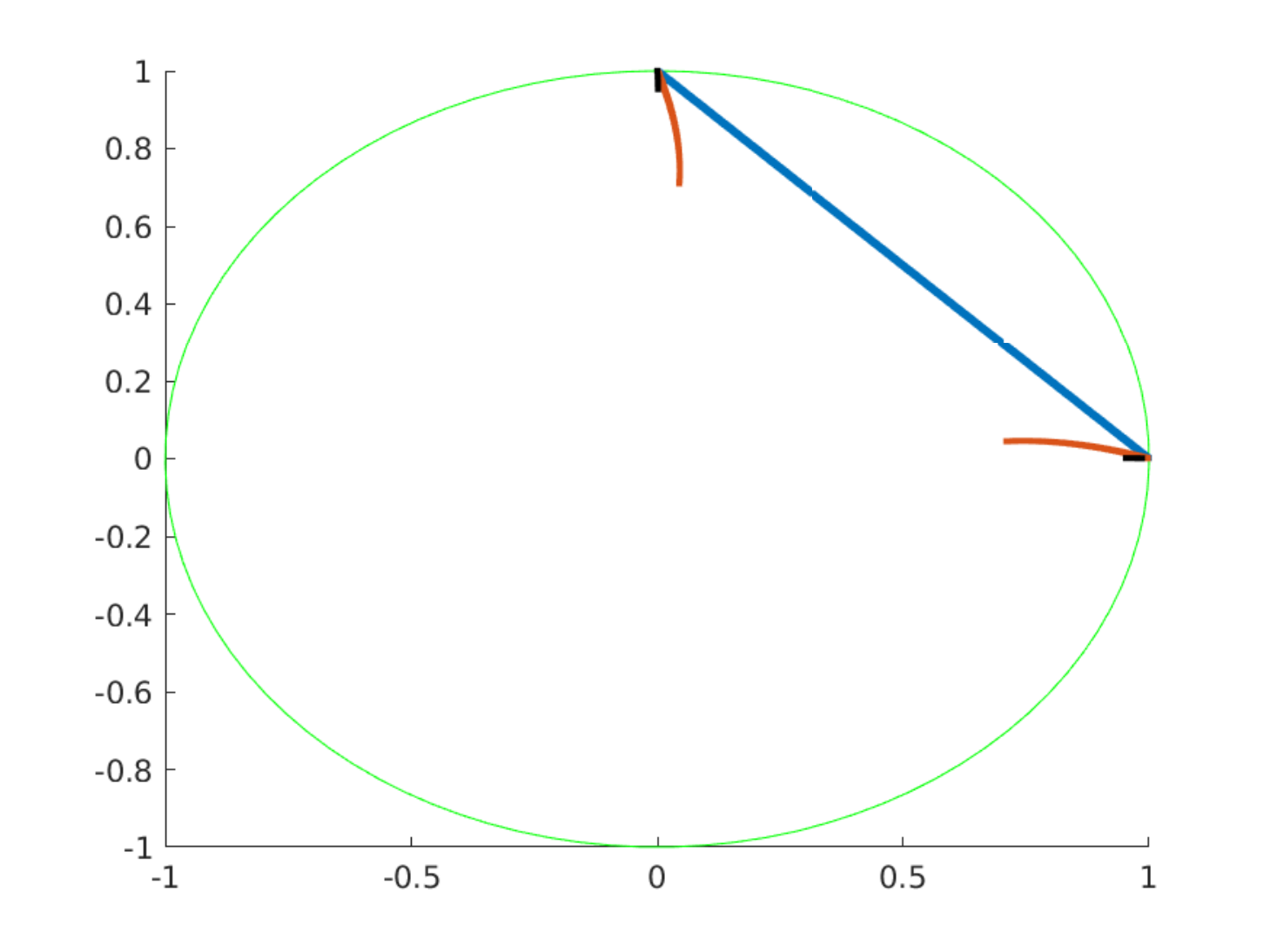}   \\ $T=\matp{   \ii & 0 \\ 0 & 1 }$
\end{center}

\begin{center}
 \includegraphics[width=200pt]{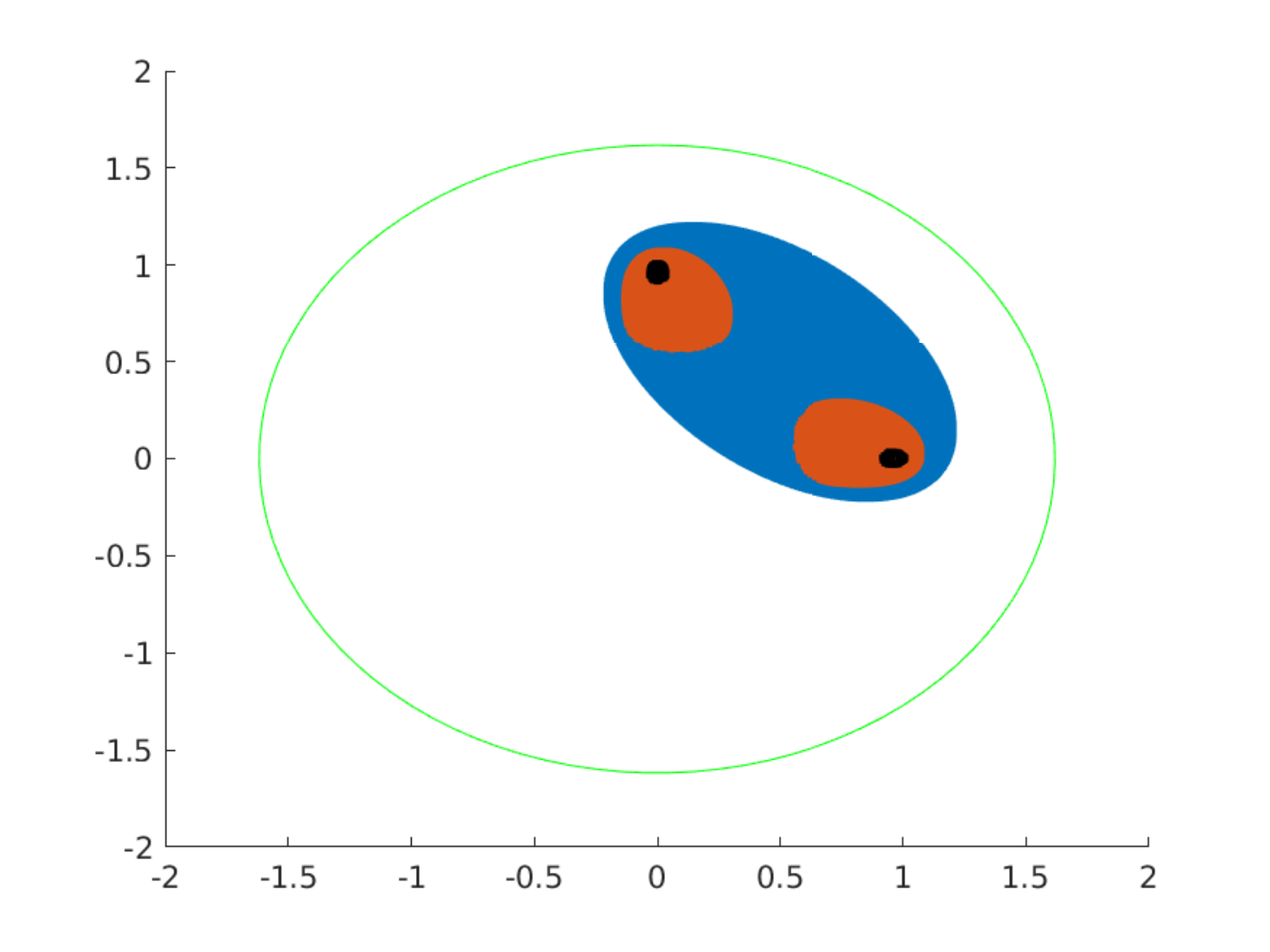} \\ $T=\matp{   \ii & 0 \\ 1 & 1 }$
\end{center}

\begin{center}
 \includegraphics[width=200pt]{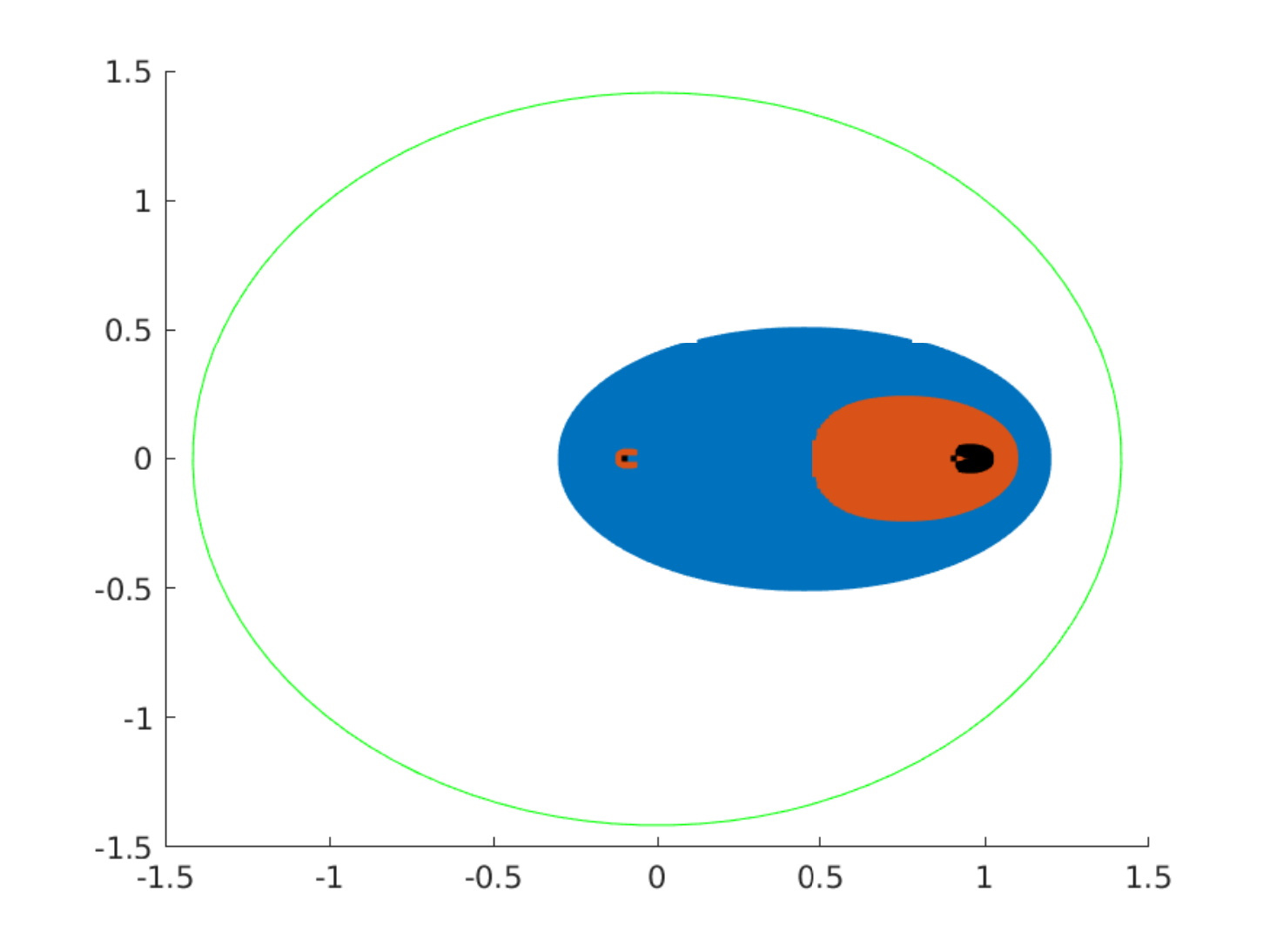} \\ $T=\matp{   -0.1 & 1 \\ 0 & 1 }$
\end{center}

\end{figure}

The next example, due to its importance and length of the argument, is presented as a proposition.

\begin{proposition}\label{2x22}
If $T=\matp{0 & 2\\ 0 &0}\in\Comp^{2,2}$, then
$$
W^\rho(T)=\frac2\rho\ \overline{\Disc}.
$$
%In consequence, if $T\in\Comp^{2,2}$ has a single eigenvalue $\lambda$ with geometric multiplicity $1$ %(i.e. $T$ has a Jordan chain) then  $W^\rho(T)$ is some disc, centred at $\lambda$.
\end{proposition}

\begin{proof}
For $ \rho=2$ the result is known, hence, assume that $\rho\neq 2$. As in Section \ref{def} we take $r=2(1-\rho^{-1})$, $r\neq 1$. Let $h=\matp{x & y}^\top\in\Comp^2$, so that
 \begin{equation*}
     \seq{Th,h}=2y\bar{x}\qquad\textup{and} \qquad \norm{Th}=|2x|.
 \end{equation*}
 The  deformed numerical range of $T$ has the following form
 \begin{equation*}
{W}^{\rho}(T)=\set{\frac{\bar{x}y}{|y|}\big( r |y|+\sqrt{ r ^2|y|^2-4( r -1)}\big) : \norm{\matp{ x\\ y}}=1, \   r ^2|y|^2-4( r -1)\geq 0    }.
 \end{equation*}
Observe that ${W}^{\rho}(T)$ is circular and since it is also convex, it is a disc centred at the origin. We prove now that its radius $\nu_\rho(T)$ equals $2-r=2/\rho$.
 Applying \eqref{dnr} we
obtain
%{\substack{m,n=-\infty,\\(m,n)\neq (0,0)}}
 \begin{equation*}\nu_\rho(T)=\sup_{\substack{\norm{h}=1,\\ \Delta_\rho(h)> 0}}\xi_\rho(h) |\seq{Th,h}|=\sup_{\substack{|x|^2+|y|^2=1,\\  r ^2|y|^2-4( r -1) > 0}}|x|( r |y|+\sqrt{ r ^2|y|^2-4( r -1)}).
 \end{equation*}
 It is a matter of elementary calculation that for $ r \neq 1$ we have $\nu_\rho(T)\leq 2- r $.

Setting
 \begin{equation}\label{sincos}
     y%=\frac{1}{2}(1+\frac{2( r -1)}{ r ^2- 2r +2})
     =\sqrt{\frac{ r ^2}{2( r ^2- 2r +2)}},\quad
     x%=\frac{1}{2}(1-\frac{2( r -1)}{ r ^2- 2r +2})
     =\sqrt{\frac{( r -2)^2}{2( r ^2- 2r +2)}},\quad h=\matp{x\\y}
 \end{equation}
 we see that in fact $\nu_\rho(T)= 2- r $.

\end{proof}

\section{The deformed numerical range of an  operator}\label{Operators}

 First, let us consider the case of a quasinilpotent operator, interesting for itself and needed later on in the proof of the inclusion $\sigma(T)\subseteq W^\rho(T)$.

\begin{proposition}\label{Pquasi} Let $T$ be a bounded and quasinilpotent but not nilpotent operator  on a Hilbert space $\h$. Then there exist a sequence $\{h_k\}_{k=0}^\infty\subset\h$ such that $\|h_k\|=1$ for $k\in\Nat$ and
\begin{equation}\label{hkl}
    \Big|\seq{\frac{Th_k}{\|Th_k\|},h_k}\Big|\to 1,\quad k\to\infty.
\end{equation}
In consequence, for any $ \rho\in[1,+\infty) $, $\dom({\xi_\rho})$ is nonempty, $0$ is an accumulation point of $W^\rho(T)$, and $\nu_\rho(T)>0$.
\end{proposition}
\begin{proof} Fix $h\in\h\setminus\set{0}$ and
define a function $f$  by
\begin{equation*}
    f(z):=\sum_{n=0}^\infty z^nT^nh, \quad z\in\Comp.
\end{equation*}
Since $T$ is quasinilpotent, we infer from the root test \cite[ page 199]{rudin1987real} that $f$ is an entire $\h$-valued function.  Observe that
\begin{equation*}
    Tf(z)=\sum_{n=0}^\infty z^nT^{n+1}h=\frac{1}{z}(f(z)-h).
\end{equation*}
Note that since $T$ is not nilpotent, $f$ is not constant and $f(z)\neq 0$ implies $Tf(z)\neq 0$. Hence,

\begin{equation}\label{qu2}
    %\Big|\seq{\frac{Tg(z)}{\|Tg(z)\|},g(z)}\Big|=
    \frac{|\seq{Tf(z),f(z)}|}{\|Tf(z)\|\|f(z)\|}=\frac{|\|f(z)\|^2-\seq{h,f(z)}|}{\|f(z)- h\|\ \|f(z)\|}.
\end{equation}
By  \cite[Theorem 3.32]{rudin1991functional} there exist a sequence $\{z_k\}_{k=0}^\infty$ such that $\lim_{k\to\infty}\|f(z_k)\|=\infty$, which gives $\|f(z_k)\|/\|f(z_k)-h\|\to1$ and
%Using the Cauchy-Schwarz inequality, we deduce from \eqref{qu1} and \eqref{qu2}
\begin{equation}\label{Thn}
\left |\seq{\frac{Tf(z_k)}{\|Tf(z_k)\|   },\frac{f(z_k)}{\|f(z_k)\|}  }\right |\to 1, \ (k\to\infty).
\end{equation}
Setting $h_k=f(z_k)/\norm{f(z_k)}$ finishes the proof of \eqref{hkl}.

Note that by \eqref{DeltaRozpisana} for a fixed $ \rho\in[1,+\infty) $ there exists $k_0$ such that $\Delta_\rho(h_k)>0$ for $k>k_0$.
For those $k$ we define $w_k:=\xi_\rho(h_k)\seq{Th_k,h_k}\in W^\rho(T)$. Note that $w_k\neq 0$ as  $\xi_\rho(h_k)>1-\rho^{-1} $ and
\begin{equation}\label{zT}
z_k\seq{Th_k,h_k}=\frac{\norm{f(z_k)}^2 - \seq{h_k, f(z_k)}}{\| f(z_k) \|^2} \to1, \quad (k\to\infty).
\end{equation}
 To finish the proof we need to  prove that  $w_k\to 0$ $(k\to\infty)$,
in the light of \eqref{zT} and since $|z_k|\to \infty$ it is enough to show  that $\xi_\rho(h_k)$ is bounded.  Observe  that
$$
\frac{\sqrt{\Delta_\rho(h_k)}}{|\seq{Th_k,h_k}|}=\frac{\sqrt{\Delta_\rho(h_k)}}{\norm{Th_k}}\cdot
\frac{ \norm{f(z_k)-h}  } {\norm{f(z_k)} |z_k| |\seq{Th_k,h_k}|}.
$$
Note that on the right hand side the first factor  converges, by \eqref{Deltadef} and \eqref{hkl}, to $\sqrt{r^2-4(r-1)}=r-2$, where $r=2-2\rho^{-1}$, and the second factor converges to $1$ by \eqref{zT}, which finishes the proof.
\end{proof}

We are able now to complete the proof of the inclusion $\sigma(T)\subseteq W^\rho(T)$, $\rho\in [1,+\infty)$, showed so far in the finite dimensional case.

\begin{proof}[Proof of Theorem \ref{basic} \eqref{sigma}]

 First we show that $\sigma_{\ap}(T)\setminus\{0\}\subseteq W^\rho(T) $. Take $\lambda\in\sigma_{\ap}(T)\setminus\{0\}$. Then there exists a sequence $\{h_n\}$ of unit vectors in $\h$   such that $\|Th_n-\lambda h_n\|\to 0$. This implies that $\seq{Th_n,h_n}\to \lambda$,  $\norm {Th_n}\to |\lambda|$
and consequently $h_n\in \dom(\xi_\rho)$, for  $n$ large  enough, and $\xi_\rho(h_n)\to 1$. Hence, $\lambda \in W^\rho(T)$.

The proof now splits into several cases.

Case 1:  $0\notin\sigma(T)$.   Recall that  $\partial(\sigma(T))\subseteq\sigma_{\ap}(T)$, see e.g. \cite[Theorem 2.5]{kubrusly2012spectral}. Consequently,
\begin{equation*}\label{s1}
\sigma(T)\subseteq \conv (\partial \sigma(T))  \subseteq \conv( \sigma_{\ap}(T))=\conv( \sigma_{\ap}(T)\setminus\set0 ) \subseteq W^\rho(T).
\end{equation*}

Case 2:  $0\in\Int(\sigma(T))$. Then
$$
\sigma(T)\subseteq \conv (\partial \sigma(T)) = \conv (\partial \sigma(T)\setminus\set0)  \subseteq \conv( \sigma_{\ap}(T)\setminus\set 0)\subseteq W^\rho(T).
$$

Case 3:  $0$ is an isolated point of $\sigma(T)$. Then, by taking the Riesz projection and applying \eqref{K} and  Case 1 we see that $\sigma(T)\setminus\set 0$ is  contained in $W^\rho(T)$. Hence, the proof of Case 3 reduces to considering  $\sigma(T)=\set0$. If  $0$ is an eigenvalue then we use Proposition \ref{basicxi}\eqref{(iv)}. If $\sigma_{\p}(T)=\emptyset$ then $T$ is a quasinilpotent, but not nilpotent operator. By  Proposition \ref{Pquasi} we have, in particular, that $0\in W^\rho(T)$.

Case 4:  $0\in\partial(\sigma(T))$ and is a non-isolated point of $\sigma(T)$. Then, by compactness of $\sigma(T)$, $0$ is a non-isolated point of $\partial\sigma(T)$. In consequence,
\begin{align*}
\sigma(T)\subseteq \conv (\partial \sigma(T)) &=    \conv (\overline{\partial \sigma(T)\setminus\set0})  \\&=
\overline{ \conv (\partial \sigma(T)\setminus\set0)}
  \subseteq \overline{\conv( \sigma_{\ap}(T)\setminus\set0)}\subseteq W^\rho(T).
\end{align*}
 \end{proof}

\section{Connection with the classes $C_\rho$ of power bounded operators}\label{NF}

%In this section we will mainly deal with the deformed numerical radius $\nu_\rho(T)=\sup_{z\in W^\rho(T)} |z|$.
First  we show the connection of our deformed numerical range and radius with the dilation theory initiated by  Sz.-N\'agy and Foia\c s in \cite{nagyfoias1966} and Durszt in \cite{durszt}.  It is known that   the operator $T\in\cB(\h)$ satisfies the following condition
\begin{equation}\tag{I${}_\rho$}\label{I}
\norm h^2 -2(1-\rho^{-1}) \RE \seq{zTh,h}+(1-2\rho^{-1})\norm{zTh}^2\geq 0, \quad h\in\h,\quad |z|\leq 1,
\end{equation}
if and only if there exists a unitary operator $U$ in some Hilbert space $\Kk$ containing $\h$ as a subspace, such that for any polynomial $p$ with $p(0)=0$ holds
$p(T)=\rho\cdot Pp(U)P^*$, where $P$ is the orthogonal projection from $\Kk$ onto $\h$.
%We note that an operator $T$ which belongs to the class $C_{\rho}$ is similar to contraction, with the condition number of the similarity less or equal then  $ \max(1,\rho)$ \cite{okubo1975constants}.
Note that it follows that the spectrum of an operator satisfying \eqref{I}  is automatically contained in the closed unit disc. We refer to \cite{NagyFoiasbook} for more details.

While for dilation theory the number $\rho$ is more natural, for technical purposes in the present paper it was much more convenient to use the parameter $r=2(1-\frac1\rho)$.  We adapt the condition \eqref{I} therefore.

\begin{lemma}
For $\rho\in[1,+\infty)$ and  $r=2(1-\frac1\rho)\in[1,2)$ condition \eqref{I} is equivalent to  the following.
\begin{equation}\tag{I${}_r$}\label{Iq}
\phi_h(t):=1- r |\seq{Th,h}| t+(r-1)\norm{ Th}^2 t^2 \geq 0,\quad t\in[0,1],\  h\in\dom(\xi_\rho).
\end{equation}
\end{lemma}

\begin{proof}
First note that for $h\in\ker T$ both conditions \eqref{I} and \eqref{Iq} are trivial. The implication \eqref{I}$\Rightarrow$\eqref{Iq} follows now by setting,  for each $h\in\dom(\xi_\rho)\setminus\ker T$,  $z=wt$ with $|w|=1$ and $\RE\seq{wTh,h}=|\seq{Th,h}|$. To see the converse, consider first the case $r\in[0,1)$. Then  $\dom(\xi_\rho)$ equals  the whole unit sphere in $\h$. Setting $t=|z|$ and using the inequality $\RE\seq{zTh,h}\leq t |\seq{Th,h}|$ we get that \eqref{Iq} implies \eqref{I} for $r\in[0,1)$. The case $ \rho=2$ is trivial. Now let $r\in(1,2)$. Observe first that for unit $h\in\h\setminus\dom(\xi_\rho)$ with $Th\neq 0$ the inequality $\phi_h(t)\geq0$ is automatically satisfied on $[0,1]$, as $\phi_h(t)$ is in such case a quadratic polynomial with the positive leading coefficient and at most one real root. In consequence, $\phi_h(t)\geq0$  on $[0,1]$  for all unit $\h$, which, again by setting $t=|z|$, is equivalent to \eqref{I}.
\end{proof}

It is well known that $T\in C_2$ if and only if $W(T)\subseteq \overline\Disc$.  We present the following generalisation.

\begin{theorem}\label{Wdysk}
Let $T$ be a bounded nonzero operator on a Hilbert space. Then $T$ has a $\rho$ dilation, i.e. $T\in C_\rho$, if and only if the deformed numerical range $W^\rho(T)$ is contained in the closed unit disc.
\end{theorem}

\begin{proof}
As before,  we will use in the proof the auxiliary parameter $r=2(1-\frac1\rho)$.
Assume first that $W^\rho(T)\subseteq \overline\Disc$,  we  show   that \eqref{Iq} is also satisfied.
The cases $\rho=1,2$ are known. Let $\rho\in(1,2)$, we fix $h\in\dom(\xi_\rho)=\set{h\in\h:\norm h=1}$. If $Th= 0$ then trivially $\phi_h(t)>0$, so we assume $Th\neq 0$. Then, $\phi_h(t)$ is a quadratic polynomial with the negative leading coefficient,  $\phi_h(0)=1$ and two different real roots
\begin{equation}\label{roots}
x_\pm=\frac{r|\seq{Th,h}|\pm  \sqrt{\Delta_\rho(h)}}{2\norm{Th}^2 (r-1)  }.
\end{equation}
Consider first the case $\seq{Th,h}\neq 0$. Note that
$$
x_+<0<1\leq x_-= |\xi_\rho(h)\seq{Th,h}|^{-1},
$$
 where the last inequality follows  by  assumption that $W^\rho(T)\subseteq \overline\Disc$.  Hence, $\phi_h(t)\geq0$ for $t\in[0,1]$.

{\color{black}
 Now take unit $h$ with $Th\neq0$,  $\seq{Th,h}= 0$. As $T\neq 0$, the set $\set{h\in \h:\seq{Th,h}=0}$  has an empty interior.  (Indeed, supposing the contrary one may take
 $h$ in the interior of $\set{h\in \h:\seq{Th,h}=0}$ and arbitrary $g\in\mathcal \h$ 
 and  note that
 $$
 0=\seq{T(h+tg),h+tg}=t(\seq{Th,g}+\seq{Tg,h}) + t^2\seq{Tg,g},\quad t\in\Real,
 $$
  getting $ \seq{Tg,g}=0$, which contradicts $T\neq 0$.)} Hence, there exists a sequence of unit vectors $h_n$ with $\seq{Th_n,h_n}\neq 0$ converging to $h$.   Note that $\phi_{h_n}(t)$ converges to $\phi_{h}(t)$ pointwise in $t$, which shows that   $\phi_h(t)\geq0$ for $t\in[0,1]$. Summarising,  we have so far showed that \eqref{Iq} holds for $ \rho\in[1,2] $.

Now let $\rho\in (2,+\infty)$ or equivalently $r\in(1,2)$, we fix $h\in\dom(\xi_\rho)$. If $Th= 0$ then trivially $\phi_h(t)>0$, so we assume $Th \neq 0$. Then $\phi_h(t)$ is a quadratic polynomial with the positive leading coefficient,  $\phi_h(0)=1$ and two  (possibly equal) real roots
given by \eqref{roots}. Since $h\in\dom(\xi_\rho)$ we have $\seq{Th,h}\neq 0$ and consequently
$$
1\leq x_-= |\xi_\rho(h)\seq{Th,h}|^{-1}<x_+.
$$
This shows that $\phi_h(t)\geq 0$ on $[0,1]$, i.e. \eqref{Iq} is satisfied.

Assume now that $T\in C_{\rho}$, i.e.  \eqref{Iq} is satisfied. It is enough to show that
$|\xi_\rho(h)\seq{Th,h}|\leq 1$ for $h\in\dom(\xi_\rho)$.  Let us fix $h\in\dom(\xi_\rho)$, as $0\in\overline\Disc$ we may assume that $\seq{Th,h}\neq 0$. The cases $\rho=1,2$ ($r=0,1$) are known, let now $r\in(0,1)$.
Then $\phi_h(t)$ is a quadratic polynomial with the negative leading coefficient,  $\phi_h(0)=1$ and two different real roots \eqref{roots}.  Note that
$$
x_+< x_-= |\xi_\rho(h)\seq{Th,h}|^{-1},
$$
and so as \eqref{Iq} is assumed we have that $x_-\geq 1$.

Let now $r\in(1,2)$. Assume that $\Delta_\rho(h)>0$. Then $\phi_h(t)$ is a quadratic polynomial with the positive leading coefficient and two different roots. Hence, $x_-<x_+$ and $x_+>0$ and \eqref{Iq} implies that $x_-\geq 1$.
\end{proof}

Immediately we get the following:

\begin{corollary}\label{infi}
We have that
\begin{equation}\label{eqniewiemy}
\nu_\rho(T)=\inf\set{{t>0}: t^{-1} T\in C_\rho}, \quad \rho\in[1,+\infty).
\end{equation}
Furthermore, $\nu_\rho^{-1}(T)T\in C_\rho$.
\end{corollary}

As this fact is crucial for our investigations, the  above proof of Theorem \ref{Wdysk} is nontrivial, and it gives some insight  in the definition of $W^\rho(T)$,  we decided to present it in full detail.

\begin{remark}
The equation \eqref{eqniewiemy} above was  remarked without proof in \cite{badea2018spectral} in the following slightly different form, namely:
\begin{equation}\label{eqniewiemy2}
\inf\set{t>0:  t^{-1} T\in C_\rho}=\sup\set{\xi_\rho(h)\seq{Th,h}:\norm h=1,\ \Delta_\rho(h)\geq 0},\quad  \rho\in[1,+\infty) .
\end{equation}
{\color{black} Due to Mathematical Reviews it can be found as well in \cite{okubo97}, see also \cite{nakazi-okubo99}  for the case $\rho\in[0,2]$.}
As the proof does not seem to be clear (especially, for $r\in(1,2)$ it is not clear if the set on the right hand side of \eqref{eqniewiemy2} is equal to $W^\rho(T)$, see Remark \ref{closures}), we have decided to show a complete proof of Theorem \ref{Wdysk} above.
\end{remark}

We also get some  basic properties of $\nu_\rho(T)$, explaining the symbol $\nu_\infty(T)$  for the spectral radius of $T$.

\begin{corollary} \label{nuprop}
The following holds for any bounded linear operator $T\neq0$ on a Hilbert space $\h$:
\begin{enumerate}[\rm (i)]
\item\label{ni} the function $[1,+\infty)\ni \rho \mapsto \nu_\rho(T)$ is nonincreasing;
\item\label{b1} $\nu(T)\leq \nu_\rho(T) \leq \norm{T}$ for  $ \rho\in[1,2] $;
\item\label{b2} $\nu_\infty(T)\leq \nu_\rho(T)\leq \nu(T)$ for $\rho\in[2,+\infty)$;
\item\label{b3} $\nu_1(T)=\norm T$;
\item\label{b4}  $\nu_\rho(T)\to \nu_\infty(T)$ with $\rho\to \infty$;
\end{enumerate}
\end{corollary}

\begin{proof}
By Corollary \ref{infi} and by monotonicity of the classes $C_\rho$ with respect to $\rho$ (e.g. \cite{NagyFoiasbook}), statement \eqref{ni} follows.

 Statement \eqref{b1} follows directly from  \eqref{ni} and Theorem \ref{basic}\eqref{sigmap}.
Statement \eqref{b2} follows directly from \eqref{ni} Theorem \ref{basic}\eqref{sigma}.

Statement \eqref{b3} is obvious. To see \eqref{b4} let us fix $h_n\in\h$ and $\rho_n\to 2$ such that $\|h_n\|=1$ and $\xi_{\rho_n}(h_n)|\seq{Th_n,h_n}|\to\lim\limits_{\rho\to 2}\nu_{\rho}(T)$. Since $\Delta_{\rho_n}(h_n)>0$, one can get $|\seq{Th_n,h_n}|-\|Th_n\|\to 0$. Thus
$$ \lim_{n\to \infty}\xi_{\rho_n}(h_n)|\seq{Th_n,h_n}|\leq \lim_{n\to \infty}\frac{1}{2}\big((2-\frac{2}{\rho_n})|\seq{Th_n,h_n}|+\frac{2}{\rho_n}\|Th_n\|\big) \leq  \nu_\infty(T).$$
\end{proof}

As a second main result of this section we show that the disc with radius $\nu_\rho(T)$ is a $\rho$--spectral set.

\begin{theorem}\label{constant1/2-q} For any bounded operator $T$ in a Hilbert space and for any polynomial $p$  we have
\begin{equation}\label{qcro}
\norm{p(T)}\leq \rho \sup_{\nu_\rho(T)\cdot\overline\Disc} |p|,\quad  \rho\in[1,+\infty) .
\end{equation}
\end{theorem}

Note that for $ \rho=2$ we get \eqref{Cro1} and for $ \rho=1$ we get the von Neumann inequality \eqref{vonN}.

\begin{proof} We fix $ \rho\in[1,+\infty) $. By Proposition \ref{basic}\eqref{alpha} we have $\nu_\rho(\nu_\rho^{-1}(T)T)\leq 1$, therefore, by Theorem \ref{Wdysk} $\nu_\rho^{-1}(T)T$ is of class $C_\rho$. Hence,  one has the inequality
$$
\norm{p(\nu_\rho^{-1}(T)T)} \leq \rho \norm{p(U)}=   \rho\ \sup_{\overline\Disc}|p|.
$$
 where $U$ is an $\rho$--unitary dilation of $\nu_\rho^{-1}(T)T$ and $p$ is any polynomial.
Substituting $p(\nu_\rho(T)z)$ for $p(z)$ we get the claim.
\end{proof}

We conclude the section with a formula for $\nu_\rho(T)$, see further in Subsection \ref{Mathias} for yet another one. For this aim we define, following
 \cite{crouzeix2017numerical} and \cite{delyondelyon1999},  two Hermitian-operator-valued measures on a circle $\xi(s)=R\exp(\ii s)$ ($s\in[0,2\pi)$):
\begin{equation}\label{twomeas2}
\mu_0(\xi(s),T)=\frac{1}{2\pi\ii}(\overline{ \xi(s)} I-T^*)^{-1}(|\xi(s)|^2I-T^*T)(\xi(s) I-T)^{-1}\frac{\xi'(s)ds}{\xi(s)},
\end{equation}
and
\begin{equation}\label{twomeas1}
\mu_1(\xi(s),T)=\frac{1}{2\pi\ii}((\xi(s) I-T)^{-1}\xi'(s) - (\overline{\xi(s)} I-T^*)^{-1}\overline{\xi'(s)})ds.
\end{equation}
For brevity we will usually omit the variable $s$ and write $\xi\in R\partial\Disc$.

\begin{proposition}\label{measure_representation}
Let $T$ be a bounded operator in a Hilbert space $\h$. Then
$$
\nu_\rho(T)=\inf \set{R\geq \nu_\infty(T) \ :\ \mu_{\rho-1}(\xi,T)\geq 0 \text{ for } \xi\in R\partial \Disc},\quad  \rho\in[1,2] ,
$$
where the Hermitian operator valued measure $\mu_{t}(\xi,T)$ is defined as
% dołożyłem \hat bo jak się wstawi za t=2 to jest kolizja oznaczeń
\begin{equation}\label{mut}
\mu_t(\xi,T):=t\mu_1(\xi,T)+(1-t)\mu_0(\xi,T),\quad t\in[0,1].
\end{equation}
\end{proposition}

\begin{proof}
First we show that $\mu_{\rho}(\xi,T)\geq 0$ on $\nu_\rho(T)\overline\Disc$, which shows the inequality '$\geq$' as $\nu_\rho(T)\geq \nu_\infty(T)$ by Corollary \ref{nuprop}. For this aim let $R=\nu_\rho(T)$ and note that the operator $R^{-1}T$ is of class $\mathcal{C}_{\rho}$ by Theorem \ref{Wdysk}. Hence, condition \eqref{I} for $ R^{-1}T$ is satisfied, in particular for $|z|=1$ and $r=2(1-\frac1\rho)$ we have
$$
\norm h^2 - r R^{-1}  \RE \seq { z Th,h} + (r-1) R^{-2} \norm {Th}^2\geq 0, \quad  h\in\h,
$$
which is equivalent to
$$
\seq{\big(2R^2 - rR(z T+\bar z T^*) + 2(r-1)   T^*T\big)h,h}\geq 0, \quad  h\in\h.
$$
Hence,
$$
\seq{(r(R\bar z(R\bar z-T)+ Rz (Rz-T^*))+2(1-r)(R^2-T^*T)\frac1R) h,h}\geq 0, \quad  h\in\h,
$$
Replacing $h$ by $(zI-T)^{-1}h$ and $R\bar z$ by $\xi(s)$ we get
$$
\mu_t(\xi,T)\geq 0, \textnormal{ for } \xi\in R \partial \Disc,
$$
with $t=\frac{r}{2-r}=\rho-1$.

Now let $\mu_t(\xi,T)\geq 0$ on $R\partial\Disc$ for some $R\geq \nu_\infty(T)$. Then $R^{-1} T$ satisfies clearly condition \eqref{I}. By analogous arguments as before for  $t=\frac{r}{2-r}$ we have
\begin{equation}\label{superharmonic}
\norm h^2 - r R ^{-1}  \RE \seq { z Th,h} + (r-1)  R ^{-2} \norm {zTh}^2\geq 0, \quad  h\in\h,
\end{equation}
where $|z|=1.$ Since $r\in [0,1]$, the above function is superharmonic. Thus the inequality \eqref{superharmonic} holds for $|z|\leq 1$. In other words $R^{-1} T$ satisfies \eqref{I} with $\rho=\frac 2{2-r}$. In consequence, $R^{-1}T$ is of class $C_{\rho}$ and $R^{-1}\leq \nu_\rho(T)^{-1}$, by Corollary \ref{infi}.
\end{proof}

\section{Monotonicity and continuity of the deformed numerical range}\label{smono}

Next we turn to the questions of monotonicity and continuity of the sets $W^\rho(T)$ with respect to the parameter $\rho$.
The latter will be understood with respect to the Hausdorff distance on complex plane
\begin{equation*}
    \hau(E,F):=\max \set{\sup_{e\in E} \dist(e,F),\sup_{f\in F} \dist(f,E)}=\max \set{\sup_{e\in E} \inf_{f\in F}|e-f|,\sup_{f\in F}\inf_{e\in E}|f-e|},
\end{equation*}
where $E,F$ are compact subsets on complex plane. We formulate now the main result on monotonicity and continuity of the mapping $\rho\mapsto \overline {W^\rho(T)}$, below $\lim_{\rho_n} \overline{W^{\rho_n}(T)}$ denotes the limit with respect to the Hausdorff metric.

\begin{theorem}\label{monocont}
For a bounded operator $T$ on a Hilbert space the following holds.
\begin{enumerate}[\rm (i)]
\item\label{mono1} If $1\leq \rho_1\leq \rho_2 $  and $0\in W^{\rho_1}(T)$  then $W^{\rho_2}(T)\subseteq W^{\rho_1}(T)$.
\item\label{mono2} If $1\leq \rho_1\leq \rho_2 $  and  $0\in\Int W^{\rho_2}(T)$ then $W^{\rho_2}(T)\subseteq W^{\rho_1}(T)$.
\item\label{mono3} If  $\h$ is finite dimensional, $1\leq \rho_1\leq \rho_2 < 2$  and $0\in\Int W(T)$ then $W^{\rho_2}(T)\subsetneq W^{\rho_1}(T)$.
\item\label{cont} The function $W: \rho\mapsto W^\rho(T)$ is continuous on $[1,2]$  with to respect to the Hausdorff metric.
\end{enumerate}
\end{theorem}

\begin{remark}
First observe that   Example \ref{pictures} shows that the assumptions on location of the zero in \eqref{mono1}--\eqref{mono3} are indispensable.
Later on, in Section \ref{s:specconst}, we will assume that $0\in\Int W(T)$, which guarantees the monotonicity for all $ \rho\in[1,+\infty) $.
\end{remark}

For the proof of the theorem we need three lemmas, the first of which is well known.

\begin{lemma}\label{l1}
 The convex hull operator
$  V\mapsto  \conv V$, acting on the family of compact subsets of $\Comp$,
satisfies a Lipschitz condition $\hau(\conv(V_1),\conv(V_2))\leq \hau(V_1,V_2)$.
\end{lemma}
\begin{proof}
Let $V_i$ ($i=1,2)$  be  compact subsets of $\Comp$ and let $\tilde x\in\conv V_1$. Then there exist $x_1,x_2,\dots,x_{n}\in V_1$ and $t_1,t_2,\dots,t_{n}\in [0,1]$ such that $t_1+\dots+t_{n}=1$ and $\tilde{x}=t_1x_1+\dots+t_{n}x_{n}$ (it is enough to take $n=3$ by the Carath\'eodory theorem). Therefore,
\begin{align*}
    \dist(\tilde{x},\conv(V_2))&\leq \dist(\tilde{x},\set{t_1y_1+\dots+t_{n}y_{n}\colon y_1,\dots,y_n\in V_2})\\
    &\leq t_1\dist(x_1,V_2)+\dots+ t_{n}\dist(x_{n},V_2)\\
    &\leq \sup_{x\in V_1}\dist(x,V_2)\leq \hau(V_1,V_2).
\end{align*}
Thus $\sup_{x\in \conv(V_1)}\dist(x,\conv(V_2))\leq \sup_{x\in V_1}\dist(x,V_2)$. Reversing the roles of $V_1$ and $V_2$ completes the proof.
\end{proof}

For subsequent reasonings we need to define the following auxiliary sets.
Let $T$ be a bounded operator on a Hilbert space, for $1\leq \rho_1\leq \rho_2 $ we  set
\begin{equation*}
V_{\rho_1,\rho_2}(T)=\overline{\set{ \xi_{\rho_1}(h) \seq{Th,h}: h\in \dom(\xi_{\rho_2})}}.
\end{equation*}
Note that the definition is correct as $\dom(\xi_{\rho_2})\subseteq\dom(\xi_{\rho_1})$ by Proposition \ref{basicxi}\eqref{(i)}.

\begin{lemma}\label{l2}
For $\rho_0\in (1,+\infty)$ the mapping
\begin{equation}
   V_{\rho_0}:  [1,\rho_0)\ni \rho \to V_{\rho,\rho_0}(T)
\end{equation} is continuous with respect to the Hausdorff metric.
\end{lemma}

\begin{proof}

Fix $\rho_0\in (1,+\infty)$ and take  $\rho_1,\rho_2\in [0,\rho_0)$ with $|\frac2{\rho_1}-\frac{2}{\rho_2}|\leq \eps$ with some $\eps>0$. For $r_i=2(1-\frac1{\rho_i})$ it means $|r_1-r_2|\leq\eps$. By Theorem \ref{basic}\eqref{alpha} we may assume that $\norm T\leq1$, so that $\norm{Th}\leq 1$ and $|\seq{Th,h}|\leq 1$ for $h\in\dom(\xi_{\rho_0})$. % such that
%\begin{equation*}
   % |q-q_0|\leq\min\set{\frac{5\eps^2}{96\norm{T}^2}, \frac{\eps}{2\norm{T}},\frac{1}{2} , 2\eps\left(\norm{T}+\frac{12}{\eps}\norm{T}^2\right)^{-1}}.
%\end{equation*}
In this setting it is clear that
%\hau(E,F):=\max \set{\sup_{e\in E} \dist(e,F),\sup_{f\in F} \dist(f,E)}
\begin{equation*}
\dist(\xi_{\rho_i}(h)\seq{Th,h},V_{\rho_j,\rho_0}(T))\leq|\xi_{\rho_1}(h)-\xi_{\rho_2}(h)|,\quad h\in\dom(\xi_{\rho_0}),\ i,j=1,2.
\end{equation*}
%and
%\begin{equation*}
%\dist(\xi_{q_2}(h)|\seq{Th,h}|,V_{q_1,q_0}(T))\leq|\xi_{q_2}(h)-\xi_{q_1}(h)||\seq{Th,h}|,
%\end{equation*}
%for any $h\in\h$.
This implies the following inequality for the Hausdorff metric
\begin{align}\label{ha1}
    \hau(V_{\rho_1,\rho_0}(T),V_{\rho_2,\rho_0}(T))&\leq \sup_{h\in\dom(\xi_{\rho_0})} |\xi_{\rho_1}(h)-\xi_{\rho_2}(h)|\\&\leq \sup_{h\in\dom(\xi_{\rho_0})}\frac{1}{2}\left(\eps+|\sqrt{\Delta_{\rho_1}(h)}-\sqrt{\Delta_{\rho_2}(h)}|\right).\notag
\end{align}
Observe that if $h\in\dom(\xi_{\rho_0})$ is such that  $\sqrt{\Delta_{\rho_1}(h)}+\sqrt{\Delta_{\rho_2}(h)}\leq \sqrt\eps$
then
\begin{equation}\label{ha2}
|\sqrt{\Delta_{\rho_1}(h)}-\sqrt{\Delta_{\rho_2}(h)}|\leq \sqrt\eps
\end{equation}
 and if $h\in\dom(\xi_{\rho_0})$ is such that  $\sqrt{\Delta_{\rho_1}(h)}+\sqrt{\Delta_{\rho_2}(h)}> \sqrt\eps$
then
\begin{equation}\label{osz}
|\sqrt{\Delta_{\rho_1}(h)}-\sqrt{\Delta_{\rho_2}(h)}|\leq\frac{|\Delta_{\rho_1}(h)-\Delta_{\rho_2}(h)|}{\sqrt \eps}\leq\frac{ |r_1^2-r_2^2|+4|r_1-r_2|}{\sqrt{\eps}} \leq 8\sqrt \eps.
\end{equation}
Taking together \eqref{ha1}, \eqref{ha2} and  \eqref{osz} and the fact that $\eps>0$ was arbitrary we get the claim.

\end{proof}

\begin{lemma}\label{l3} If $1\leq \rho_1<\rho_2<\rho_3$ and $0\in\conv V_{\rho_1,\rho_3}$ then $\conv V_{\rho_2,\rho_3}\subseteq \conv V_{\rho_1,\rho_3}$.
\end{lemma}

\begin{proof}
It is enough to show that any $\lambda_0$ of the form
$\lambda_0=\xi_{\rho_2}(h_0)\seq{Th_0,h_0}$ with $h_0\in\dom({\xi_{\rho_3}})$ belongs to
$\conv{V_{\rho_1,\rho_3}(T)}$. % we may assume that  $\lambda_0=\xi_{q_2}(h_0)\seq{Th_0,h_0}$ for some $h_0\in\dom({\xi_{q_2}})$.
If $\lambda_0=0$ then trivially $\lambda_0\in \conv{V_{\rho_1,\rho_3}(T)}$, so we assume that $\lambda_0\neq 0$ and hence $\seq{Th_0,h_0}\neq 0$ and $\xi_{\rho_1}(h_0) \geq \xi_{\rho_2}(h_0) \geq 1-\frac1{\rho_2}>0$.
Note that %$h_0\in\dom(\xi_{q_1})$ by Proposition \ref{basicxi}(i) and consequently
$\lambda_1=\xi_{\rho_1}(h_0)\seq{Th_0,h_0}\in V_{\rho_1,\rho_3}(T)$ and $\lambda_0=\frac{\xi_{\rho_2(h_0)}}{\xi_{\rho_1(h_0)}}\lambda_1$.     As we assumed that $0\in \conv{V_{\rho_1,\rho_3}(T)}$, we get $\lambda_0\in \conv{V_{\rho_1,\rho_3}(T)}$.
\end{proof}

\begin{proof}[Proof of Theorem \ref{monocont}]

\eqref{mono1}  The proof follows the same lines as the proof of Lemma \ref{l3} with $\conv{V_{\rho_i,\rho_3}(T)}$ replaced by $W^{\rho_i}(T)$, $i=1,2$ and $h_0\in\dom(\xi_{\rho_2})$ (however, the statement itself is not a direct consequence of Lemma \ref{l3}).

\eqref{mono2} Assume that $0\in\Int W^{\rho_2}(T)$. We show that $0\in \overline{W^{\rho_1}(T)}$, which will finish the proof.
Consider the set
$$
Q:=\set{\rho\in[1,\rho_2): 0\in\Int (\conv{V_{\rho,\rho_2}(T))}}.
$$
As the mapping $V_{\rho_2}:[1,\rho_2)\ni \rho\mapsto \conv{V_{\rho,\rho_2}(T)}$ is,  by Lemmas \ref{l1} and \ref{l2}, continuous, the set $Q$ is nonempty and open in $[1,+\infty)$. Furthermore, by Lemma \ref{l3} we see that if $\rho\in Q$ then $[\rho,\rho_2)\subseteq Q$.
Hence, to show that $Q$ is closed in $[1,\rho_2)$ it is enough to take a decreasing sequence $Q\ni q_n\searrow \rho$ and show that $\rho\in Q$. Applying once again  Lemma \ref{l3} we have that $\conv V_{q_{n},\rho_2}\subseteq \conv V_{q_{n+1},\rho_2}$ and therefore, the distance of $0$ to  $\partial(\conv V_{q_{n},\rho_2})$ is bounded from below by some $\eps>0$. By continuity of $V_{\rho_2}$ we get
$0\in \conv V_{\rho,\rho_2}$, which shows $Q=[1,\rho_2)$. To finish the proof note that  since $\rho_1\in Q$, we have
$0\in\conv V_{\rho_1,\rho_2}\subseteq W^{\rho_1}(T)$.

\eqref{mono3} In view of \eqref{mono1} and \eqref{mono2} it is enough to show that $W^{\rho_1}(T)\neq W^{\rho_2}(T)$. Take $\lambda_2\in\partial W^{\rho_2}(T)\setminus\sigma(T)$. Then $\lambda_2=\xi_{\rho_2}(T)\seq{Th,h}$ for some $h\in\h$. By Proposition \ref{basicxi}\eqref{(ii)} we have $\xi_{\rho_1}(h)>\xi_{\rho_2}(h)$.  Hence, $\lambda_1=\xi_{\rho_1}(h)\seq{Th,h}\notin W^{\rho_2}(T)$, as otherwise $\lambda_2$ would lie in the interior of $W^{\rho_2}(T)$ due to $0\in\Int W(T)\subseteq W^{\rho_2}(T)$. Clearly, $\lambda_1\in W^{\rho_1}(T)$.

\eqref{cont} The statement follows directly from Lemmas \ref{l1}, \ref{l2} and the fact that
$\conv V_{\rho,1}(T)=W^{\rho}(T)$ for $ \rho\in[1,2] $. %  To prove \eqref{semicont} observe that

\end{proof}

\section{Spectral constants of the deformed numerical range}\label{s:specconst}

Following Crouzeix \cite{crouzeix2004bounds} we define
$$
\Psi_\Omega(T):=\sup\set{\norm{f(T)} : f\in\h^\infty(\Omega), \ \norm{f}_{L^\infty(\Omega)\leq 1   }},
$$
where $\Omega$ is an open, non empty, convex subset of $\Comp$ and $\h^\infty(\Omega)$ denotes the Hardy space. Note that
$\Psi_\Omega(T)<+\infty$, provided that $\sigma(T)\subseteq \Omega$, due to the Cauchy integral formula.
Furthermore, we define
$$
\Psi_\rho(T):=\sup\set{\norm{f(T)} : f \text{ polynomial}, \ \norm{f}_{L^\infty(W^\rho(T))\leq 1   }}, \quad  \rho\in[1,+\infty) .
$$
The following result was proved as \cite[Lemma 2.2]{crouzeix2004bounds} for $ \rho=2$, we show that it is true for all $ \rho\in[1,2] $ under the additional assumption that $0\in\Int W(T)$. Note that while for $ \rho=2$ this assumption may be simply omitted, due to the law $W(T+\alpha I)=W(T)+\alpha$, $\alpha\in\Comp$, we cannot drop this assumption for $\rho<2$.

\color{black}
\begin{lemma}\label{spectral}
Let $T$ be a matrix. Then for any closed  set $V$ containing the numerical range $W(T)$  we have
\begin{equation}\label{VV}
\sup\set{\norm{f(T)} : f \text{ polynomial}, \ \norm{f}_{L^\infty(V)\leq 1   }}=\sup\set{\Psi_\Omega(T): \Omega\supseteq V,\  \Omega\text{ open }}
\end{equation}
In particular, if $0\in\Int W(T)$ and $ \rho\in[1,2] $ then
$$
\Psi_\rho(T)=\sup\set{\Psi_\Omega(T): \Omega\supseteq W^\rho(T),\  \Omega\text{ open }}.
$$
\end{lemma}

\begin{proof}
Let $\lambda$ be an eigenvalue of $T$ lying on the boundary of $V$. As $W(T)\subseteq V$ the eigenvalue $\lambda$ lies also on the boundary of $W(T)$. Now, the   proof follows the same lines as in \cite[Lemma 2.2]{crouzeix2004bounds}.
%
%
%Let $ \rho\in[1,2] $. It is enough to show that each eigenvalue on the boundary of $W^\rho(T)$ is on the boundary of $W(T)=W_1(T)$, then the proof follows the same lines as in \cite[Lemma 2.2]{crouzeix2004bounds}.
%Let $\lambda\in\partial W^\rho(T)$ be an eigenvalue of $T$. Clearly,
%$\lambda\in W(T)$ and due to $0\in\Int W(T)$ we have
% \begin{equation}\label{tlambda}
% t \lambda\in W(T),\quad t\in[0,1].
% \end{equation}
% On the other hand, $t\lambda\notin W^\rho(T)$ for $t>1$, as $\lambda$ is on the boundary of $W^\rho(T)$, $0$ is in the interior of $W^\rho(T)$ and $W^\rho(T)$ is convex. Hence, by Theorem \ref{monocont} \eqref{mono2}, we have that $t\lambda\notin W(T)$ for $t>1$, which together with \eqref{tlambda} implies that $\lambda$ is on the boundary of $W(T)$.

\end{proof}
\color{black}

\begin{remark}
For the sake of completeness let us mention that it is also possible to show that for a finite dimensional $\h$ the second statement of Lemma \ref{spectral} holds for $\rho\in[1,2+\eps]$, where $\eps>0$ depends on $T$.  The proof is, however, rather technical and the result will not be used later on.
\end{remark}

 %Furthermore, if $\h$ is finite dimensional,
%the function $\Omega\mapsto \Psi_\Omega(T)$ is either strictly decreasing or constantly equal one.
%Combining this information with Lemma \ref{spectral} and Theorem \ref{monocont} \eqref{cont} we get the following result.

Note that in general the function $T\mapsto \Psi_1(T)$ is not continuous, as the example of the matrix $T=\matp{0 & \eps \\ 0 & 0}$ with $\eps\geq 0$ shows. Nonetheless, the function $\Psi_\rho(T)$ is continuous with respect to $\rho$ for fixed $T$:

\begin{theorem}\label{qcont}
Let $T$ be a bounded operator on a Hilbert space with $\ 0\in\Int W(T)$, then the function
\begin{equation}\label{function}
[1,2]\ni \rho\mapsto \Psi_\rho(T)
\end{equation}
 is continuous and increasing.
 If, additionally,     $\h$ is finite dimensional, then the function in \eqref{function} is either constantly equal to $1$ or strictly increasing.
\end{theorem}

\begin{proof}
The proof is based on Lemma 2.2 from \cite{crouzeix2004bounds}. It was shown therein  that for fixed $T$ the  function $\Omega\mapsto \Psi_\Omega(T)$ is decreasing and continuous with respect to the  Hausdorff metric, note that this part of the proof did not use the fact that $\h$  is finite dimensional.
Combining this information with Lemma \ref{spectral} and Theorem \ref{monocont} \eqref{cont} we get the desired continuity in \eqref{function}.

Monotonicity in \eqref{function} results directly from Theorem \ref{monocont}\eqref{mono2}.
To prove that the function in question is strictly increasing if $\Psi_2(T)>1$  we  use Theorem \ref{monocont}\eqref{mono3}. In the light of this result, combined again with Lemma \ref{spectral}, it is enough to remark that  $\Psi_{\Omega_1}(T)< \Psi_{\Omega_2}(T)$ if $\Omega_2\subsetneq\Omega_1$ and $\Psi_2(T)>1$, see again \cite{crouzeix2004bounds}, Lemma 2.2.

\end{proof}

\begin{remark}
If the  condition $0\in\Int W(T)$ is not satisfied, it is tempting to replace $T$ by $T-\alpha I$, where $\alpha$ is chosen apropriateltely, e.g., it is the barycentre of the spectrum of $T$. However, one should be careful with such manipulations, as $W^\rho(T)$ is not translable, see Remark \ref{trans} and Example \ref{pictures}. Hence,  the corresponding spectral constants  $\Psi_\rho(T)$ and $\Psi_\rho(T-\alpha I)$ may differ.
\end{remark}

We are also able now to complete the analysis from Proposition \ref{2x22}.
\begin{corollary} For $T=\matp{0 & 2\\ 0 & 0}$ we have
$$
\Psi_\rho(T)=\rho,\quad  \rho\in[1,+\infty) .
$$
\end{corollary}

\begin{proof}
Take any $ \rho\in[1,+\infty) $. By Proposition \ref{2x22} we have that $W^\rho(T)=\frac2\rho\overline\Disc$. Considering the polynomial $p(z)=z$ we get $\Psi_q(T)\geq \rho $. On the other hand Theorem \ref{constant1/2-q} gives us the opposite inequality.
\end{proof}

Recall that the Crouzeix conjecture  says that  $\Psi_2(T)\leq 2$ for any bounded operator $T$,  (equivalently: for any matrix $T$, see \cite{Crouzeix2007}). Note the following corollary from our considerations above.

\begin{corollary} The Crouzeix conjecture does not hold if and only if there exists a matrix $T$ with $0\in\Int W(T)$ and $\rho\in[1,2)$ such that $\Psi_\rho(T)=2$.
\end{corollary}

\begin{proof} Suppose that there exists a matrix $T$ with $\Psi_2(T)>2$. Observe that  $W(T)$ has a nonempty interior, otherwise $T$ is an affine transformation of a Hermitian matrix and $\Psi_2(T)=1$, contradiction.  Due to $\Psi_2(T+\alpha I)=\Psi_2(T)$ for $\alpha\in\Comp$, one can find a matrix $T$ with $0\in\Int W(T)$ and $\Psi_2(T)>2$. Application of Theorem \ref{qcont} finishes the proof of the forward implication. The converse implication follows directly from the last statement of  Theorem \ref{qcont}.
\end{proof}

\section{Relations to other  concepts}

\subsection{Dilation result}\label{Mathias}

In \cite{mathias1994induced} Mathias and Okubo showed that
$$
\nu_\rho(T)= \nu(B_\rho\otimes T), \quad B_\rho=\mat{cc} 0 & 2\sqrt{\frac{2-\rho}\rho} \\ 0 & \frac{2(\rho -1)}\rho \rix, \quad \rho\in[1,2],
$$
i.e. the sets $W^\rho(T)$ and $W(B_\rho\otimes T)$ have the same maximal absolute value.
We show the following.
\begin{theorem}
For any complex square matrix $T$ we have
$$
W^\rho(T)\subseteq W(B_\rho\otimes T), \quad  \rho\in[1,2] .
$$
\end{theorem}

\begin{proof}
Let us take a point $\xi_\rho(h)\seq{Th,h}\in W^\rho(T)$ with $\norm h=1$. Without loss of generality, multiplying $T$ by a constant if necessary,  we may assume that $\seq{Th,h}>0$, so that
$$
\xi_\rho(h)\seq{Th,h}=\frac12\left( r\seq{Th,h}+\sqrt{r^2 \seq{Th,h}^2+4(1-r)\norm{Th}^2  }  \right ),
$$
where $r=2-\frac2\rho$.
Note that by the formula for the numerical range of a $2\times 2$ matrix (cf. e.g. \cite[Chapter 1]{GustafsonRao}) applied to
$$
X=\mat{cc}  0 & 2 \sqrt{1-r}\norm{Th} \\ 0 & r\seq{Th,h}\rix,
$$
we have $\xi_\rho(h)\seq{Th,h}=\nu(X)$.  Hence, there exist a unit vector $x=[x_1,x_2]^\top\in\Comp^2$, $x_2\neq 0$, such that
\begin{equation}\label{cccc}
\xi_\rho(h)\seq{Th,h}=\seq{Xx,x}=2 \sqrt{1-r}\norm{Th}\bar x_1x_2 +r\seq{Th,h}|x_2|^2.
\end{equation}
We set
$$
f=\frac{  x_1\bar x_2  }{|x_2|\norm{Th}}Th,\quad g=|x_2|h.
$$
Note that  $F=[f,g]^\top$ is a unit vector in $\Comp^{2n}$. Furthermore,
 $$
 \seq{  (B_\rho\otimes T)  F,F} = 2\sqrt{\frac{2-\rho}\rho}\seq{Tg,f}+\frac{2(\rho -1)}\rho \seq{Tg,g},
 $$
 which is the right hand side of \eqref{cccc}. In consequence, $\xi_\rho(h)\seq{Th,h}\in W(B_\rho\otimes T)$.

\end{proof}

\begin{remark}
Note that if $T$ is Hermitian, then $W^\rho(T)$ is contained in the real line, while $B_\rho\otimes T$ is usually not Hermitian. Hence, the converse inclusion clearly does not hold in general. Also, recall that  $W(B_\rho)\cdot W(T)\subseteq W(B_\rho\otimes T)$, but, in general, $W^\rho(T)$ is not a subset of $W(B_\rho)\cdot W(T)$ due to $\nu(B_\rho)=1$.
\end{remark}
Example \ref{rqex} at the end of the paper contains a plot of the sets $W^\rho(T)$ and $W(B_\rho\otimes T)$ for a non-normal matrix $T$.

\subsection{$q$-numerical range} Let $T$ be a matrix. The {\em $q$-numerical range} $W(T:q)$ is defined as
$$
W(T:q)=\set{\seq{Th,k}: \norm h=\norm k=1,\ \seq{h,k}=q},\quad q\in[0,1].
$$
It is known that the $q$-numerical range is convex and that
\begin{equation}\label{Tsing}
W(T:q)=\set{q\seq{Th,h}+\sqrt{1-q^2}w\sqrt{\norm{Th}^2-|\seq{Th,h}|^2}:\norm h=1,\ |w|\leq 1}
\end{equation}
 (cf. \cite{tsing1984constrained}).  Clearly, $q^{-1}W(T:q)\supseteq W(T)$, in particular \eqref{VV} holds for $\Omega=q^{-1}W(T:q)$.  In fact, by Theorem 2.7 of \cite{li1994generalized}, all the eigenvalues of $T$ are in the interior of $q^{-1}W(T:q)$ and \eqref{VV} holds simply by continuity of the Cauchy integral.
  Furthermore, we have the following inclusion.
 \begin{theorem}
Let $T$ be a complex square matrix and let $q\in(0,1]$, $\rho\in[1,+\infty)$. Then the inclusion
\begin{equation}\label{WrhoWq}
W^\rho(T)\subseteq q^{-1}W(T:q)
\end{equation}
hold, provided at least one of the following conditions is satisfied
\begin{enumerate}[\rm (i)]
\item $0< q\leq\sqrt2/2$; \label{simple}
\item $\sqrt2/2<q\leq 1$ and $2 q^2 \leq \rho \leq \frac{2q^2}{2q^2-1}$;\label{compli}
\item $0\in  W(T)$ and $2q^2\leq \rho$.\label{withW}
\end{enumerate}
 \end{theorem}

 \begin{proof}
 As $W(T:q)$ is closed and convex, it is enough to show that $\xi_\rho(h)\seq{Th,h}\in q^{-1}W(T:q)$ for arbitrary unit vector $h\in\dom(\xi_\rho)$,  let us fix such $h$. Without loss of generality, multiplying $T$ by a constant if necessary,  we may assume that $\seq{Th,h}>0$.

First let us consider the cases \eqref{simple} and \eqref{compli} simultaneously.
  By formula \eqref{Tsing}, it is enough to show the inequalities
 \begin{equation}\label{toshow}
  \alpha-\sqrt{q^{-2}-1}\sqrt{\beta^2-\alpha^2}\leq\frac12\left(  r  \alpha+\sqrt{ r  ^2 \alpha^2+4(1- r )\beta^2  }\right)\leq \alpha+\sqrt{q^{-2}-1}\sqrt{\beta^2-\alpha^2},
  \end{equation}
 where $0<\alpha:=\seq{Th,h}\leq\norm{Th}=:\beta$ and  $r=2(1-\rho^{-1})$.
 To see  the right inequality of \eqref{toshow} note that $\rho\geq 2q^2$ in both cases \eqref{simple} and \eqref{compli}, due to $\rho\geq 1$. This can be written as
 $$
 1-r\leq q^{-2}-1.
 $$
 Multiplying both sides by $\beta^2-\alpha^2\geq 0$ and adding
 $(2-r)\alpha \sqrt{(q^{-2}-1)(\beta^2-\alpha^2)}\geq0$ to the right hand side we obtain
 $$
 \frac{\alpha^2r^2}4  + (1-r)\beta^2\leq \left(1-\frac r2\right)^2\alpha^2+  (q^{-2}-1)(\beta^2-\alpha^2) + (2-r)\alpha \sqrt{(q^{-2}-1)(\beta^2-\alpha^2)}.
 $$
 Note that the left hand side above is nonnegative, due to the assumption $h\in\dom(\xi_\rho)$. Taking square roots and adding $r\alpha/2$ to both sides, we get the desired right hand side of \eqref{toshow}.

 Now let us show the left hand side of \eqref{toshow}. First observe that
\begin{equation}\label{qr}
r-1\leq q^{-2}-1.
\end{equation}
 Indeed, in case  \eqref{simple} one simply has $r-1\leq 1 \leq q^{-2}-1$, while in  case \eqref{compli} one has $2q^2-1>0$, and  hence, the inequality \eqref{qr} follows by simple manipulations of $\rho\leq 2q^2/(2q^2-1)$. Multiplying both sides of \eqref{qr} by $\beta^2-\alpha^2\geq 0$ one obtains
 $$
 \alpha^2\left(1-\frac r2\right)^2\leq (q^{-2}-1)(\beta^2-\alpha^2)+ \frac{r^2\alpha^2}4+(1-r)\beta^2.
 $$
 Hence,
 $$
 \alpha\left(1-\frac r2\right)\leq \sqrt{(q^{-2}-1)(\beta^2-\alpha^2)}+ \frac12\sqrt{{r^2\alpha^2}+4(1-r)\beta^2},
 $$
 where the expression under the second square root is nonnegative due to $h\in\dom(\xi_q)$. This finishes the proof of \eqref{toshow} and the proof in cases \eqref{simple} and \eqref{compli}.

Consider now the case \eqref{withW}. Then, as $0\in W(T)\subseteq q^{-1}W(T:q)$, one needs to prove only the right hand inequality of \eqref{toshow}. This follows the same lines as before, starting from  $\rho\geq 2q^{2}$.
 \end{proof}

Example \ref{rqex} at the end of the paper contains a plot of the sets $W^\rho(T)$ and $q^{-1}W(T:q)$ for a non-normal matrix $T$.

\subsection{Normalised numerical range}

For a bounded operator $T$ on a Hilbert space the normalised numerical range is defined as
$$
F_N(T)=\set{ \frac{\seq{Tx,x}}{\norm{Tx}\norm x}:  \norm{Tx}\neq 0    }.
$$
Although the definition seems to be similar to $W^\rho(T)$ (especially if $\rho=1$), there seems to be no clear relation as with the $q$-numerical range or $W(B_\rho\otimes T)$ as above. However, while proving the spectral inclusion in Section \ref{Operators} we have contributed to the following result.

\begin{proposition}
For any nonzero bounded operator $T$ on a Hilbert space the closure of $F_N(T)$ contains some point on the unit circle.
\end{proposition}

\begin{proof}
It was shown in Proposition 7 of \cite{gevorgyan2006some}, see also \cite[Proposition 1.2]{spitkovsky2017normalized}, that the result holds if there is a nonzero point in the approximative spectrum of $T$. Hence, the only case left to consider is the case when
the approximative spectrum equals $\set0$, i.e., $T$ is a quasinilpotent operator.

If $T\neq 0$ is  nilpotent, then there exist two unit vectors $f_2,f_1$, not necessarily orthogonal, such that $Tf_2=\alpha f_1$, $\alpha\in\Comp\setminus\set0$,  $Tf_1=0$. Then observe that with $t>0$
$$
\frac{|\seq{T(f_1+t f_2),(f_1+tf_2)}|}{\norm{T(f_1+t f_2)}\norm{f_1+tf_2}}=\frac{(t+t^2\seq{f_1,f_2})|\alpha|}{t|\alpha|\sqrt{1+t^2} }   \to 1,\quad t\to 0.
$$

The remaining case of a qusinilpotent, but not nilpotent operator, is the essence of Proposition \ref{Pquasi} above.
\end{proof}

\subsection{Davies-Wielandt shell}\label{ss:shell}

In this subsection we will present some facts about the set
$$
V_\rho(T):=\set{\xi_\rho(h)\seq{Th,h}:\norm h=1, \ \seq{Th,h}\neq 0},\quad \rho\in[1,2].
$$
Note that by definition our deformed numerical range equals $\overline{\conv}(V_{\rho}(T)\cup\set 0)$ if $0\in W(T)$ and    $\overline{\conv}(V_{\rho}(T))$
in the opposite case. Also let us recall that the Davies-Wielandt shell is defined (\cite{MR243373,MR73558}) as
$$
DW(T):=\set{(\RE\seq{ T h,h},\IM\seq{ T h,h},\norm{Th}^2):\norm h=1}\subseteq\Real^3.
$$
See, e.g., \cite{MR2440672} for basic properties, one of which is that $DW(T)$ is convex  if $\dim \h\geq 3$.

We will follow the idea of presenting the set under investigation as a continuous image of the the Davies-Wielandt shell, applied in
\cite{lins2018normalized} to the normalised numerical range $F_N(T)$. In particular, it was shown in \cite{lins2018normalized} that $F_N(T)$ is path connected and simply connected. See also \cite{chien2002davis} for the connection with the $q$-numerical range $W(T:q)$. However, the situation with the set $V_\rho(T)$ is slightly more technical.

\begin{proposition}\label{Vshape} Let $T$ be a bounded operator on a Hilbert space $\h$ with $\dim \h\geq 3$, assume that $\alpha T+\beta$ is not selfadjoint  for all $\alpha,\beta\in\Comp$. Then for all $\rho\in[1,2]$ the set $V_\rho(T)$ defined above is path-connected. Furthermore, $V_\rho(T)$ is disjoint with the open disc 
 $\sqrt{2\rho^{-2}-1}\ \sigma_{\min}(T)\Disc$, where $\sigma_{\min}(T)$ denotes the smallest singular value of $T$.
\end{proposition}

\begin{proof}
Observe that for $\rho\in[1,2]$ the set $V_\rho(T)$
is a continuous image of the following subset of the Davies-Wielandt shell
$$
DW_0(T):=\set{(\RE\seq{ T h,h},\IM\seq{ T h,h},\norm{Th}^2):\norm h=1,\ \seq{Th,h}\neq0}.
$$
Hence,  $V_\rho(T)$ is path-connected, provided that $DW_0(T)$ is path connected. The latter follows from the following reasoning.
As $DW(T)$ is convex, it is either  contained in an affine hyperplane or has a nonempty interior in $\Real^3$.
In the former case the projection on the first two variables is a homeomorphism. Note that  $W(T)$ is the image of this projection and it has, by assumption on $T$,  a nonempty interior in $\Real^2$. Thus $W(T)\setminus \set0$, and consequently $DW_0(T)$, are path connected. In the latter case, i.e. if $DW(T)$ is homeomorphic with a closed ball in $\Real^3$, path connectedness of $DW_0(T)$ follows from a simple topological reasoning.

The following estimate for arbitrary unit $h$ shows the second claim
$$
|\xi_\rho(h)\seq{Th,h}|\geq \sqrt{1-r}\norm{Th}\geq \sqrt{1-r}\ \sigma_{\min}(T),\quad r=2\left(1-\frac1\rho\right).
$$
\end{proof}

\begin{example} Let $T$ be a Hermitian invertible matrix, having both positive and negative eigenvalues. By Theorem
 \ref{basic}\eqref{0}  the set $V_\rho(T)$ ($\rho\in[1,2]$) is contained in the real line. By Proposition  \ref{basicxi}\eqref{(iv)} it contains the eigenvalues. Hence, by Proposition \ref{Vshape},  $V_\rho(T)$ is not connected.
 \end{example}

The following Example shows a typical situation, where the set $V_\rho(T)$ is connected, but not simply connected. We illustrate also the here the results from previous subsections.

\begin{example}\label{rqex} Consider the matrix
$T=\matp{-1& 0& 0 \\ 0& 1& \ii \\ 0& 1& 0 }   $ and the parameters $q=0.9$, $\rho=2q^2=1.62$. In Figure \ref{pictures7} one may find the eigenvalues (black stars) and the plots of $V_\rho(T)$ (green), $W(B_\rho\otimes T)$ (blue), and $q^{-1}W(T:q)$ (red).
The red circle is of radius $\nu_\rho(T)$, the magenta circle is of radius $\norm T$, both centred at the origin. Note the following additional facts:
\begin{itemize}
\item $-1$ is a normal eigenvalue lying on the boundary of $W^\rho(T)$, cf. Lemma \ref{spectral}.
\item Decreasing the parameter $\rho$ leads to extending the set $W^\rho(T)$, for the values of $\rho$ near $1$ this eigenvalue is no longer on the boundary of  $W^\rho(T)$.
\item The hole in the green set $V_\rho(T)$ is due to Proposition \ref{Vshape}, here $\sqrt{1-r}\ \sigma_{\min}(T)\simeq 0.2993$.
\item One may observe that $W(B_{2q^2}\otimes T)\subseteq q^{-1}W(T:q)$, which is frequent situation.  However, this is not true in general, e.g. for $T=\matp{0 & 2\\ 0 & 0}$ the inclusion does not hold.
\end{itemize}

\begin{figure}
\begin{center}
 \includegraphics[width=400pt]{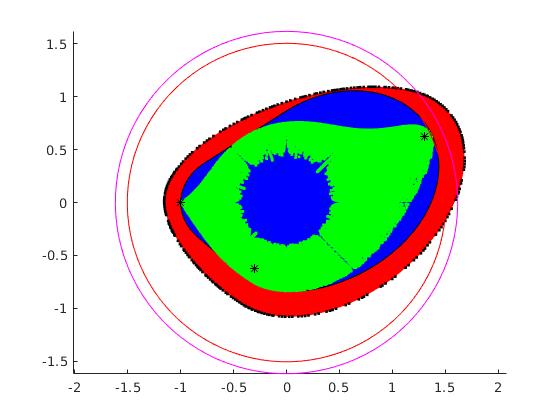}
\end{center}
\caption{Figure for Example \ref{rqex}.  \label{pictures7}}
\end{figure}

\end{example}

\section*{Acknowledgment} The authors would like to express their gratitude to the anonymous referee. The interesting questions posed in the report resulted in the last section of the paper.

\bibliographystyle{plain}
\bibliography{Literatura}

\end{document}